\documentclass[12pt]{article}

\usepackage{graphicx} 
\usepackage{amsfonts,amsmath,amssymb,amsthm}
\usepackage{hyperref}
\usepackage{tikz,float,tikz-cd}

\usepackage{color,soul}

\newtheorem{theorem}{Theorem}[section]
\newtheorem{proposition}[theorem]{Proposition}

\newtheorem{definition}[theorem]{Defintion}
\newtheorem*{remark}{Remark}

\def\multiset#1#2{\ensuremath{\left(\kern-.3em\left(\genfrac{}{}{0pt}{}{#1}{#2}\right)\kern-.3em\right)}}
\usepackage{chngcntr} 
\counterwithin{figure}{section}

\usepackage{caption}
\usepackage{subcaption}

\usepackage{xcolor}
\usepackage[dvipsnames]{xcolor}

\hypersetup{
    colorlinks=true,
    linkcolor=magenta,
    citecolor=blue
}

\title{Mirror Variations of Catalan Multijections}
\author{Muhammad Adam Dombrowski \thanks{Email: \href{mailto:muhammad_adam_dombrowski@brown.edu}{\tt muhammad\_adam\_dombrowski@brown.edu}}}
\date{\today}

\begin{document}

\maketitle

\begin{abstract}
	A Catalan multijection is a way to partition a set of $\binom{2n}{n}$ objects into equivalence classes of size $n+1$, where each equivalence class has exactly one distinct Catalan object. Isaak and Langley investigated when two multijections produced the same partitions in an unpublished manuscript. We propose some new multijections for which we get different partitions from previously discovered multijections.
\end{abstract}

\section{Introduction}

The Catalan numbers, $\{1, 1, 2, 5, 14, 42...\}$, are an integer sequence (\href{https://oeis.org/A000108}{A000108}) that appears in many places: triangulations of $n$-gons, binary strings, and lattice paths that stay below the diagonal. We denote the $n$-th Catalan number by $C_n$. There are multiple formulas for this integer sequence, including recurrence relations, but the most relevant for us will be the following formula given by the Chung-Feller Theorem.

\begin{equation*}
    C_n=\frac{1}{n+1}\binom{2n}{n}
\end{equation*}

Before stating the theorem, we would like to define what we mean by lattice paths. The following rigorous definition comes from Stanley\cite{stanley2015catalan}.

\begin{definition}[Lattice Paths in $\mathbb{Z}^d$]
    Let $S$ be a subset of $\mathbb{Z}^d$. A lattice path in $\mathbb{Z}^d$ of length $k$ with steps in $S$ is a sequence $v_0,v_1,\dots, v_k\in\mathbb{Z}^d$ such that each consecutive difference $v_i-v_{i-1}$ lies in $S$.
\end{definition}

For the purposes of this paper, we only need to consider lattice paths in $\mathbb{Z}^2$, as well as the consecutive difference $v_i-v_{i-1}$ being either $(1,0)$ or $(0,1)$. Additionally, we will consider lattice paths of the previous form that run from $(0,0)$ to $(n,n)$. We will adopt the following narrower definition for lattice paths in the context of this paper.

\begin{definition}
    Let $n\in\mathbb{N}$. A lattice path in $\mathbb{Z}^2$ of length $2n$ is a sequence $p_1p_2\dots p_n$, where each $p_i\in \{(1,0),(0,1)\}$, or in this specific case, denoted $R$ for right-step $(1,0)$ and $U$ for up-step $(0,1)$. In general, we call $p_i$ the steps of a lattice path.
\end{definition}

\begin{remark}
    Let $p,q$ be steps in a lattice paths. Then $p$ and $q$ are equivalent, denoted by $p\cong q$, if they both correspond to the same type of step; if $p,q$ are both up-steps then $p\cong q$, and if $p$ is an up-step and $q$ is a down-step, then $p\ncong q$.
\end{remark}

There are two common ways to visualize these lattice paths. The first is as we have stated, a lattice path consists of steps $(1,0)$ and $(0,1)$ from the origin to $(n,n)$. However, we only acknowledge and will make use of an alternative definition which considers steps of $(1,1)$ or $(1,-1)$, creating a path from $(0,0)$ to $(n,0)$. The former will be referred as a lattice path in \textit{staircase notation} and the latter as a lattice path in \textit{mountain notation}. 

\begin{remark}
    A lattice path in staircase notation will be a sequence $p_1p_2\dots p_n$ from $(0,0)$ to $(n,n)$ where each $p_i$ is either $(1,0)$, denoted as a right-step $r$, or $(0,1)$, denoted as an up-step $u$. A lattice path in mountain notation will be a sequence $p_1p_2\dots p_n$ from $(0,0)$ to $(n,0)$ where each $p_i$ is either $(1,1)$, denoted as an up-step $u$ or $(1,-1)$, denoted as a down-step $d$.
\end{remark}

We also define an equivalence of lattice paths.

\begin{definition}\label{latticepathequivalence}
    Let $P,Q$ be lattice paths. Then $P$ and $Q$ are equivalent, denoted by $P\cong Q$, if $P$ and $Q$ are the same length, say $n$, and $p_i\cong q_i$ for $i=1,\dots, n$. 
\end{definition}

The number of lattice paths from $(0,0)$ to $(n,n)$ is not counted by the Catalan numbers. However, if we impose a certain condition on these lattice paths however, then we get something that is counted by $C_n$.

\begin{definition}
    Let $P$ be a lattice path in staircase notation. Then, $P$ is a Dyck path if it stays below the diagonal $y=x$.
\end{definition}

If a lattice path is in mountain notation, then a Dyck path will be a lattice path which does not go below the $x$-axis. This follows a bit less intuitively from the given definition in staircase notation, but the reason for this note is to align our setup with a certain proof which defined the Dyck path with respect to this parameter. The reader can see that this labeling of ``above'' the $x$-axis, ``below'' the diagonal, and its corresponding labeling of the diagonal is arbitrary; the important part is consistently choosing this labeling in order to properly define Dyck paths in context.

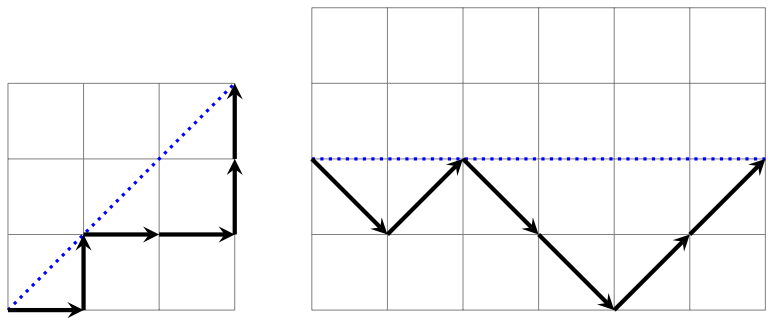
\begin{figure}

\begin{center}
\begin{tikzpicture}

\draw[step=1cm,gray,very thin] (0,0) grid (3,3);
\draw[blue, dotted, very thick] (0,0) -- (3,3);
\draw[-stealth,black,ultra thick] (0,0) -- (1,0);
\draw[-stealth,black,ultra thick] (1,0) -- (1,1);
\draw[-stealth,black,ultra thick] (1,1) -- (2,1);
\draw[-stealth,black,ultra thick] (2,1) -- (3,1);
\draw[-stealth,black,ultra thick] (3,1) -- (3,2);
\draw[-stealth,black,ultra thick] (3,2) -- (3,3);
\end{tikzpicture}
\qquad
\begin{tikzpicture}
\draw[step=1cm,gray,very thin] (0,-2) grid (6,2);
\draw[blue, dotted, very thick] (0,0) -- (6,0);
\draw[-stealth,black,ultra thick] (0,0) -- (1,-1);
\draw[-stealth,black,ultra thick] (1,-1) -- (2,0);
\draw[-stealth,black,ultra thick] (2,0) -- (3,-1);
\draw[-stealth,black,ultra thick] (3,-1) -- (4,-2);
\draw[-stealth,black,ultra thick] (4,-2) -- (5,-1);
\draw[-stealth,black,ultra thick] (5,-1) -- (6,0);
\end{tikzpicture}
\end{center}
\caption{Staircase notation (left) and mountain notation (right).}
\label{notation}
\end{figure}

Notice how the lattice path drawn in Figure \ref{notation} does not cross the diagonal. The lattice path drawn in Figure \ref{notdyckpath} is not an example of a Dyck path.

\begin{figure}
\begin{center}
\begin{tikzpicture}
\draw[step=1cm,gray,very thin] (0,0) grid (3,3);
\draw[gray, dotted, very thick] (0,0) -- (3,3);
\draw[-stealth,black,ultra thick] (0,0) -- (1,0);
\draw[-stealth,black,ultra thick] (1,0) -- (1,1);
\draw[-stealth,black,ultra thick] (1,1) -- (2,1);
\draw[-stealth,black,ultra thick] (2,1) -- (2,2);
\draw[-stealth,black,ultra thick] (2,2) -- (2,3);
\draw[-stealth,black,ultra thick] (2,3) -- (3,3);
\end{tikzpicture}
\end{center}
\caption{This lattice path is not a Dyck path.}
\label{notdyckpath}
\end{figure}
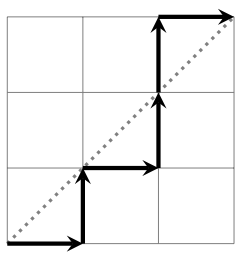

Now, we are ready to give the Chung-Feller Theorem. Originally proved by Chung and Feller in a slightly different form in \cite{chungfeller}, the following statement is still equivalent.

\begin{theorem}[Chung-Feller Theorem]\label{chungfeller}
	The number of lattice paths consisting of $n$ up-steps and $n$ right-steps which stay below the diagonal, or Dyck paths with $n$ up-steps and $n$ right-steps is $\frac{1}{n+1}\binom{2n}{n}$.
\end{theorem}

There are many proofs of the Chung-Feller Theorem, ranging from analytic to combinatorial. The $\frac{1}{n+1}$ in the theorem suggests that one way to prove this is to show that a larger set of objects of size $\binom{2n}{n}$ can be divided into subsets of $n+1$ subsets of equal size, with each subset containing a Catalan object. 

For example, in the case of general lattice paths from $(0,0)$ to $(n,n)$, there are $2n$ total steps, and all that is requires is to choose where the $n$ right-steps will go. Then, by the Chung-Feller Theorem, then $n+1$ divides our total number of lattice paths to give us the total number of Dyck paths.

Another perspective of this mapping is that we can divide into subsets of size $n+1$, but have one Dyck path per subset. For now, we can loosely call these subsets ``Catalan partitions''. While this labeling may sound novel, plenty of combinatorial Chung-Feller proofs induce this partitioning. Finally, we can start to ask when two of these kind of proofs are equivalent.

\section{Defining Multijections}

Isaak and Langley provided a framework that characterizes this approach with the following two definitions. The following two defintions are attributed to Isaak from \cite{isaakmultijections}.

\begin{definition}[Multijection]\label{multijection1}
    Assume that $A$ and $\mathcal{A}$ are finite sets. We will say that a map $\alpha$ from $A$ to $\mathcal{A}$ is a $t$-uniform multijection if the preimages $\alpha^{-1}(a)$ for $a\in\mathcal{A}$ all have size $t$. For a given function, the relation $xRy$ if $f(x)=f(y)$ is trivially an equivalence relation so the preimages partition the domain $A$. Then, since all have the same size, $t$. we get $|\mathcal{A}|=\frac{1}{t}|A|$ for a $t$-uniform multijection.
\end{definition}

One can see the resemblance of a combinatorial Chung-Feller proof under this broader notion of a multijection. From here on, we will use the term ``multijection'' when describing other combinatorial Chung-Feller proofs.

Within the context of lattice paths, we can imagine that $A$ is our set of general lattice paths of length $2n$, and $\mathcal{A}$ is the set of Dyck paths of length $2n$. Let $\alpha$ be a multijection that partitions $A$ into our Catalan partitions, each partition mapping to a distinct member $a\in\mathcal{A}$. To recover our partitions, then we simply consider the preimage $\alpha^{-1}(a)$. For context, suppose that $\alpha$ is the exceedance multijection. The left-most lattice path in Figure \ref{equivalenceclass1} is a Dyck path by our definition, or our element of $\mathcal{A}$. Then all four lattice paths lie in the preimage of the left one under $\alpha$, and they consist one Catalan partition.

With this definition, we can define when we expect two multijections to be equivalent.

\begin{definition}\label{equivalentmultijections}
    Let $\alpha$ be a $t$-uniform multijection from $A$ to $\mathcal{A}$ and $\beta$ a $t$-uniform multijection from $B$ to $\mathcal{B}$ with $|A|=|B|$ and hence also $|\mathcal{A}|=|\mathcal{B}|$. Let $\sigma$ be a bijection from $A$ to $B$.

    We will say that $\alpha$ and $\beta$ are translations of each other with respect to $\sigma$ if, under $\sigma$ the preimages induce the same partition. That is, if $\alpha(a_1)=\alpha(a_2)$ implies $\beta(\sigma(a_1))=\beta(\sigma(a_2))$ and hence also $\beta(b_1)=\beta(b_2)$ implies $\alpha(\sigma^{-1}(b_1))=\alpha(\sigma^{-1}(b_2))$.
\end{definition}

One note is that for the purposes of the paper, a multijection will always be $n+1$ uniform for us. These definitions give us a rigorous footing to describe equivalent proofs, but we wish to give an intuitive approach. A multijection is a way of partitioning a larger set of objects in Catalan equivalence classes. Intuitively, two multijections are the same when they induce the same partitions. 

While the concept of a \textit{multijection} is newer, there is plenty of literature involving a multijection approach, such as \cite{callan1995pair}, \cite{woan2001uniform}. These papers usually incorporate two things: a uniformly distributed parameter, and an additional bijection. The following defintion comes from Woan \cite{woan2001uniform}.

\begin{definition}
    A uniformly distributed parameter $\gamma$ is a parameter which is uniformly distributed on the values $[1,\dots,n]$.
\end{definition}

Technically, a uniformly distributed parameter sounds more complicated than it needs to be. The most relevant example for this paper would be ``flaws'', counted in staircase notation as the number of ``up-steps'' above the diagonal, and in mountain notation as the number of ``up-steps'' above the $x$-axis. The name ``flaw'' can be assigned to the fact that lattice paths that contain 0 flaws are Dyck paths. Other examples can be found in \cite{woan2001uniform}. In this context, we will say that $\alpha$ is a multijection with respect to uniformly distributed parameter $\gamma$ if the partitions of the multijection are dictated by $\gamma$.

This means that for a given $n$, considering lattice paths of length $2n$ between $(0,0)$ and $(n,n)$. Then the sets $S_k$, where $S_k$ consists of $2n$ length lattice paths with $k$ flaws satisfies. 

\begin{equation}\label{bijectionintro}
    |S_0|=|S_1|=\dots = |S_{n-1}|=|S_n|
\end{equation}

The reader may wonder if these constitute the partitions from our multijection, and the short answer is that they do not. For example, if $k$ denotes the number of flaws, then $S_0$ would be the set of Dyck paths. In order to align with our multijection definition, then there must only be one Dyck path per partition.

We gave the definition of uniformly distributed parameter earlier, but the equality in \ref{bijectionintro} is not immediately obvious in general. Multijection literature focuses on proving a bijection between the sets $S_k$ and $S_{k+1}$. If this is proven, then all sets $S_k$ have same size, and the multijection partitions can finally be constructed.

Let $\phi_k:S_k\to S_{k+1}$ be the bijection between $S_k$ and $S_{k+1}$. This bijection uniquely links an element of $S_k$ with $S_{k+1}$. However, for each element of our Dyck path set, there is a unique chain of elements through $S_k$, each linking unique elements in each set along the way. More importantly, if we chain $\phi_k$, then for an element in $S_k$, there is a unique element in $S_0$ that it maps to. This suggests a perfect case to apply our multijection framework to.

In all future multijections setups in this paper, we assume that we begin with the larger set $L_n$ of lattice paths between $(0,0)$ and $(n,n)$ of length $2n$ consisting of $n$ up-steps, and $n$ right-steps. There is some parameter $\gamma$, which dictates the partition of $L_{2n}$ into sets $S_0,S_1,\dots S_k$. where $k$ is the amount of ``$\gamma$''s a lattice path contains. For the paper, we only really consider when $\gamma$ are the flaws of a Dyck path. $S_0$ becomes our Catalan set $C_{2n}$. In this paper, $\gamma$ is strictly flaws, but it should be possible to generalize this further to other uniformly distributed parameters.

\begin{proposition}\label{setup}
    For $k=0,1,\dots,n$ let $S_k$ be the subset of lattice paths with $k$ flaws and suppose that a proof $P$ proves bijections $\phi_k:S_k\to S_{k+1}$. Then we say the multijection $\alpha:L_{2n}\to C_{2n}$ is defined by the following: for $p\in S_k$, then $\alpha(p)=\phi_0^{-1}\dots\phi_{k-1}(\phi_k^{-1}(p))$.
\end{proposition}

\begin{proof}
    Throughout this write-up, we interchange between $S_0$ and $C_{2n}$, since a zero-flaw lattice path would be a Dyck path, and it would be in our Catalan set. We would need to show that for all $\tilde{p}\in C_{2n}$, that the preimage of $\tilde{p}$, $\alpha^{-1}(\tilde{p})$ all have equal size, particularly $n+1$. Callan's paper showed that the number of flaws of a lattice path are uniformly distributed, and exhibited that the sets $|S_0|=|S_1|=\dots =|S_k|$.

    What does the preimage of $\tilde{p}\in C_{2n}$ look like? By construction, the preimage is the following set

    \begin{equation*}
        \{\tilde{p},\phi_0(\tilde{p}),\phi_1(\phi_0(\tilde{p})),\dots,\phi_{n-1}(\dots \phi_1(\phi_0(\tilde{p})))\}
    \end{equation*}

    We can verify this by applying $\alpha$ to each element. For $\tilde{p}$, $\alpha$ simply sends to itself. For $\phi_0(\tilde{p})$, since $\phi_0$ is a bijection between $S_0$ to $S_1$, then $\phi_0(\tilde{p})$ is an element of $S_1$. Since $\phi_0(\tilde{p})$ is also unique, i.e. there will not be another element of $S_0$. Let's check that $\phi_0(\tilde{p})$ maps to $\tilde{p}$. From our definition of $\alpha$, we get

    \begin{equation*}
        \alpha(\phi_0(\tilde{p}))=\phi_0^{-1}(\phi_0(\tilde{p}))=\tilde{p}
    \end{equation*}

    This process is repeatable for each element. For example, $\phi_1(\phi_0(\phi_0(\tilde{p})))$ is an element of $S_2$, and so forth. Since $\tilde{p}$ was arbitrary, then this construction works for all $\tilde{p}\in C_{2n}$. The bijections imply that all elements are mapped to a unique element of $C_{2n}$. With this setup, then each preimage have size $n+1$. Since they have the same size, we are done.
    
\end{proof}

One nice thing about this approach is that it allows us to analyze Chung-Feller Theorem proofs by these partitions of lattice paths, rather than a top-down approach that only consider the proof method in itself. The multijection framework states that two multijections are the same if the partitions are the same. This gives us some motivation to find Chung-Feller Theorem proofs that create unique partitions.

\begin{remark}
    Proposition \ref{setup} proves that the multijection framework can be carried over to many combinatorial Chung-Feller Theorem proofs.
\end{remark}

With that, we have covered the preliminaries of multijections. We can proceed to apply this theory to proofs of the Chung-Feller Theorem which involve a permutation of paths.

\section{Path-swapping}

Young-Ming Chen published a proof of the Chung-Feller Theorem on lattice paths with respect to ``flaws'' of a Dyck path by proving that the sets with $k$ flaws were in bijection with one another, i.e. the equality described in \ref{bijectionintro}. He did this by decomposing a lattice path in $S_k$ into multiple subpaths based on specific steps in the path, then permuting the individual components to get a path in $S_{k+1}$. Then, he showed that this permutation was reversible. The general method here is clear; take a lattice path, then swap around parts of the path to get a new path, and show this transformation is well-defined and unique. For the purposes of this paper, we will call this path-swapping.

There are two types of multijections in literature which involve path-swapping: single-step and multiple-step.

\subsection{Single parameter path-swapping}

First, we present a definition for single-step parameter swapping.

\begin{definition}
    We call a path-swapping multijection a single-step path-swapping multijection if the bijection between the sets $S_k$ and $S_{k+1}$.
\end{definition}

There are plenty of examples of single-step parameter swapping. We will focus on a bijection presented by Callan in \cite{callan1995pair}. We will walkthrough his argument in mountain notation, but transition to staircase notation afterwards. Callan's notion of a flaw corresponds to our staircase notion of a flaw.

In reverse, we demonstrate a map from $S_{k+1}\to S_k$ first. To be clear, we want to take a lattice path with $k+1$ flaws, and return a lattice path with $k$ flaws, and show this process is reversible. For this demonstration, let $k=2$. We denote the $x$-axis at $y=0$ in blue. Draw horizontal lines above the $x$-axis, and label the up-steps from bottom row to top row, left to right, starting at $1$. An example is given in \ref{callan1}. 

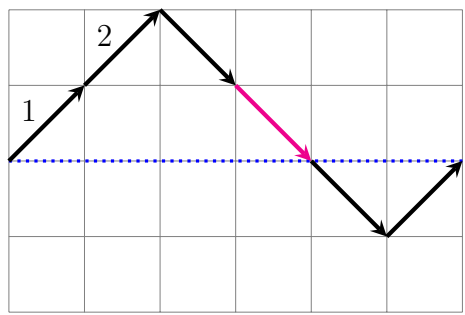
\begin{figure}
    \centering
    \begin{tikzpicture}
        \centering
        \draw[step=1cm,gray,very thin] (0,-2) grid (6,2);
        \draw[blue, dotted, very thick] (0,0) -- (6,0);
        \draw[-stealth,black,ultra thick] (0,0) -- node[above left=-1pt, yshift=-3pt,text=black, font=\normalsize] {1} (1,1);
        \draw[-stealth,black,ultra thick] (1,1) -- node[above left=-1pt, yshift=-3pt,text=black, font=\normalsize] {2} (2,2);
        \draw[-stealth,black,ultra thick] (2,2) -- (3,1);
        \draw[-stealth,magenta,ultra thick] (3,1) -- (4,0);
        \draw[-stealth,black,ultra thick] (4,0) -- (5,-1);
        \draw[-stealth,black,ultra thick] (5,-1) -- (6,0);
    \end{tikzpicture}
    \caption{A lattice path with 2 flaws.}
    \label{callan1}
\end{figure}

Then, we identify the first down-step following up-step 1. This down-step is marked in magenta. It is worth noting that another way to think of this critical down-step is the first down-step touching the $x$-axis after the path crosses above the $x$-axis for the first time. This will be helpful for translating Callan's proof into staircase notation.

Now, swap the portions of the path before and after the highlighted down-step. The path in \ref{callan1} is turned into the path depicted in \ref{callan2}.

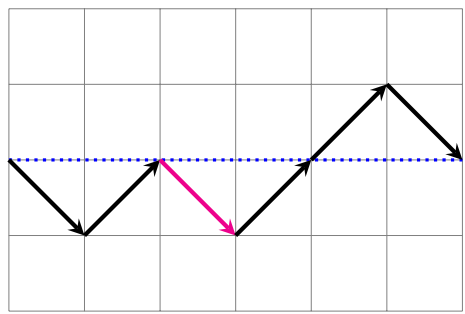
\begin{figure}
    \centering
    \begin{tikzpicture}
        \centering
        \draw[step=1cm,gray,very thin] (0,-2) grid (6,2);
        \draw[blue, dotted, very thick] (0,0) -- (6,0);
        \draw[-stealth,black,ultra thick] (0,0) -- (1,-1);
        \draw[-stealth,black,ultra thick] (1,-1) -- (2,0);
        \draw[-stealth,magenta,ultra thick] (2,0) -- (3,-1);
        \draw[-stealth,black,ultra thick] (3,-1) -- (4,0);
        \draw[-stealth,black,ultra thick] (4,0) -- (5,1);
        \draw[-stealth,black,ultra thick] (5,1) -- (6,0);
    \end{tikzpicture}

    \caption{From \ref{callan1}, the resulting lattice path with $1$ flaw.}
    \label{callan2}
\end{figure}

In order to reverse this process, we begin with the path in \ref{callan2}. We draw horizontal lines below the $x$-axis, and scan from right to left, starting in the first row below the $x$-axis and going down, identifing each up-step in chronological order. Similar to the above, we identify the down-step directly after the first up-step, in \ref{callan3}. We note two things; this down-step is the same down-step involved in the path-swapping demonstrated above, and that this down-step is the first down-step after the first up-step to touch the $x$-axis.

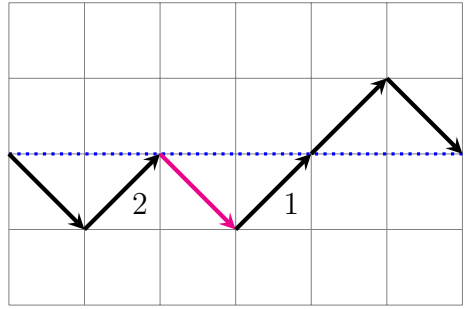
\begin{figure}
    \centering
    \begin{tikzpicture}
        \centering
        \draw[step=1cm,gray,very thin] (0,-2) grid (6,2);
        \draw[blue, dotted, very thick] (0,0) -- (6,0);
        \draw[-stealth,black,ultra thick] (0,0) -- (1,-1);
        \draw[-stealth,black,ultra thick] (1,-1) -- node[below right=-1pt, yshift=3pt,text=black, font=\normalsize] {2} (2,0);
        \draw[-stealth,magenta,ultra thick] (2,0) -- (3,-1);
        \draw[-stealth,black,ultra thick] (3,-1) -- node[below right=-1pt, yshift=3pt,text=black, font=\normalsize] {1} (4,0);
        \draw[-stealth,black,ultra thick] (4,0) -- (5,1);
        \draw[-stealth,black,ultra thick] (5,1) -- (6,0);
    \end{tikzpicture}

    \caption{Reversing the process.}
    \label{callan3}
\end{figure}

Thus, being able to reverse this process completes the bijection, and Callan concludes that the Chung-Feller Theorem is proved in this method. The proof is explained in more detail in Callan's original paper; we simply wanted to review this proof in the context of multijections. In \ref{action}, we switch from mountain notation to staircase notation.

\begin{figure}

\begin{subfigure}{0.5\textwidth}
\begin{center}
\begin{tikzpicture}
\draw[step=1cm,gray,very thin] (0,0) grid (3,3);
\draw[blue, dotted, very thick] (0,0) -- (3,3);
\draw[-stealth,black,ultra thick] (0,0) -- (0,1);
\draw[-stealth,black,ultra thick] (0,1) -- (0,2);
\draw[-stealth,black,ultra thick] (0,2) -- (1,2);
\draw[-stealth,magenta,ultra thick] (1,2) -- (2,2);
\draw[-stealth,black,ultra thick] (2,2) -- (2,3);
\draw[-stealth,black,ultra thick] (2,3) -- (3,3);
\end{tikzpicture}
\end{center}
\end{subfigure}
\begin{subfigure}{0.5\textwidth}
\begin{center}
\begin{tikzpicture}
\draw[step=1cm,gray,very thin] (0,0) grid (3,3);
\draw[gray, dotted, very thick] (0,0) -- (3,3);
\draw[-stealth,black,ultra thick] (0,0) -- (1,0);
\draw[-stealth,black,ultra thick] (1,0) -- (1,1);
\draw[-stealth,magenta,ultra thick] (1,1) -- (2,1);
\draw[-stealth,black,ultra thick] (2,1) -- (2,2);
\draw[-stealth,black,ultra thick] (2,2) -- (2,3);
\draw[-stealth,black,ultra thick] (2,3) -- (3,3);

\end{tikzpicture}
\end{center}
\end{subfigure}

\caption{\centering Callan's proof in staircase notation. Notice how we can ``rotate'' a mountain notation lattice paths to get staircase notation lattice paths in this case.}
\label{action}
\end{figure}
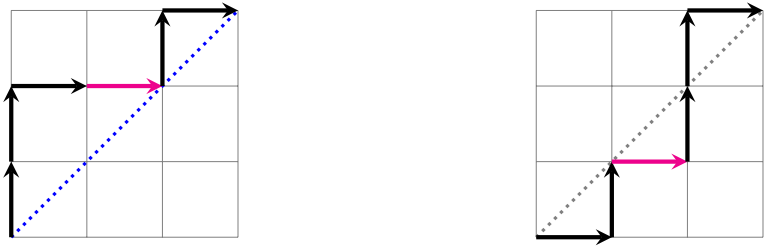

Now that we have described the proof, we can consider the multijection that corresponds to this proof. Rather than deal with the definition of a multijection, we consider the partitions induced by the proof. In Figure \ref{callanpartitions}, we show the partition for $n=3$, meaning a lattice path of length $2n$. Each row denoted $S_k$, with the top being our Dyck paths of length $2n$ and $S_0$ by convention of flaw. The following row is $S_1$, and so forth downwards. Each column represents one partition in the multijection induced by Callan's proof. As one can see, there is exactly one Dyck path in each column, and one element of $S_k$ in each column for each $k=0,\dots, n$.

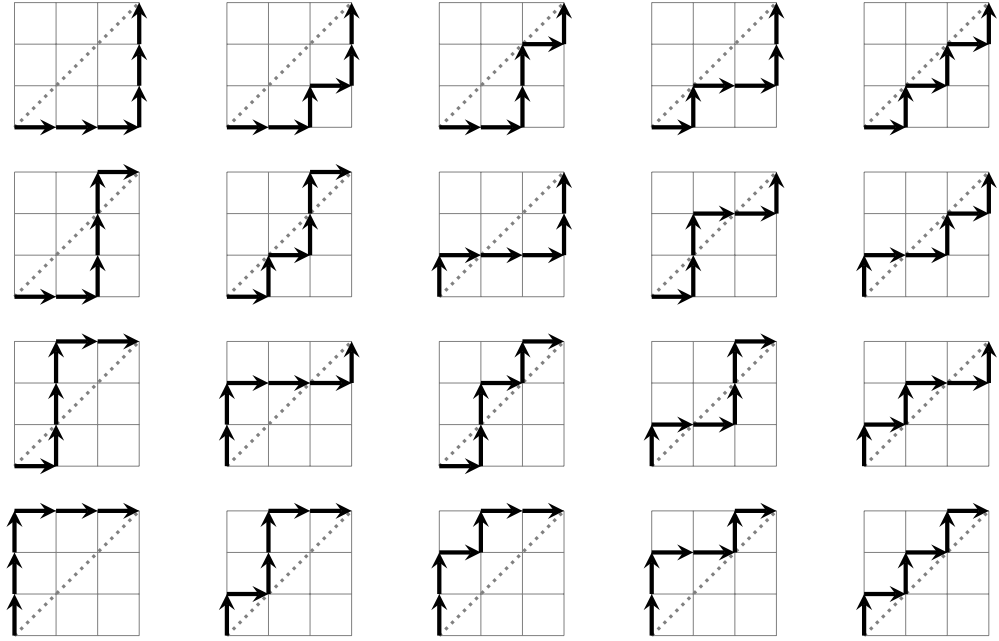
\begin{figure}
\centering
\begin{subfigure}{0.18\textwidth}
\centering
\begin{tikzpicture}[scale=0.55]
\draw[step=1cm,gray,very thin] (0,0) grid (3,3);
\draw[gray, dotted, very thick] (0,0) -- (3,3);
\draw[-stealth,black,ultra thick] (0,0) -- (1,0);
\draw[-stealth,black,ultra thick] (1,0) -- (2,0);
\draw[-stealth,black,ultra thick] (2,0) -- (3,0);
\draw[-stealth,black,ultra thick] (3,0) -- (3,1);
\draw[-stealth,black,ultra thick] (3,1) -- (3,2);
\draw[-stealth,black,ultra thick] (3,2) -- (3,3);
\end{tikzpicture}
\end{subfigure}
\hfill
\begin{subfigure}{0.18\textwidth}
\centering
\begin{tikzpicture}[scale=0.55]
\draw[step=1cm,gray,very thin] (0,0) grid (3,3);
\draw[gray, dotted, very thick] (0,0) -- (3,3);
\draw[-stealth,black,ultra thick] (0,0) -- (1,0);
\draw[-stealth,black,ultra thick] (1,0) -- (2,0);
\draw[-stealth,black,ultra thick] (2,0) -- (2,1);
\draw[-stealth,black,ultra thick] (2,1) -- (3,1);
\draw[-stealth,black,ultra thick] (3,1) -- (3,2);
\draw[-stealth,black,ultra thick] (3,2) -- (3,3);
\end{tikzpicture}
\end{subfigure}
\hfill
\begin{subfigure}{0.18\textwidth}
\centering
\begin{tikzpicture}[scale=0.55]
\draw[step=1cm,gray,very thin] (0,0) grid (3,3);
\draw[gray, dotted, very thick] (0,0) -- (3,3);
\draw[-stealth,black,ultra thick] (0,0) -- (1,0);
\draw[-stealth,black,ultra thick] (1,0) -- (2,0);
\draw[-stealth,black,ultra thick] (2,0) -- (2,1);
\draw[-stealth,black,ultra thick] (2,1) -- (2,2);
\draw[-stealth,black,ultra thick] (2,2) -- (3,2);
\draw[-stealth,black,ultra thick] (3,2) -- (3,3);
\end{tikzpicture}
\end{subfigure}
\hfill
\begin{subfigure}{0.18\textwidth}
\centering
\begin{tikzpicture}[scale=0.55]
\draw[step=1cm,gray,very thin] (0,0) grid (3,3);
\draw[gray, dotted, very thick] (0,0) -- (3,3);
\draw[-stealth,black,ultra thick] (0,0) -- (1,0);
\draw[-stealth,black,ultra thick] (1,0) -- (1,1);
\draw[-stealth,black,ultra thick] (1,1) -- (2,1);
\draw[-stealth,black,ultra thick] (2,1) -- (3,1);
\draw[-stealth,black,ultra thick] (3,1) -- (3,2);
\draw[-stealth,black,ultra thick] (3,2) -- (3,3);
\end{tikzpicture}
\end{subfigure}
\hfill
\begin{subfigure}{0.18\textwidth}
\centering
\begin{tikzpicture}[scale=0.55]
\draw[step=1cm,gray,very thin] (0,0) grid (3,3);
\draw[gray, dotted, very thick] (0,0) -- (3,3);
\draw[-stealth,black,ultra thick] (0,0) -- (1,0);
\draw[-stealth,black,ultra thick] (1,0) -- (1,1);
\draw[-stealth,black,ultra thick] (1,1) -- (2,1);
\draw[-stealth,black,ultra thick] (2,1) -- (2,2);
\draw[-stealth,black,ultra thick] (2,2) -- (3,2);
\draw[-stealth,black,ultra thick] (3,2) -- (3,3);
\end{tikzpicture}
\end{subfigure}

\vspace{0.5cm}

\centering
\begin{subfigure}{0.18\textwidth}
\centering
\begin{tikzpicture}[scale=0.55]
\draw[step=1cm,gray,very thin] (0,0) grid (3,3);
\draw[gray, dotted, very thick] (0,0) -- (3,3);
\draw[-stealth,black,ultra thick] (0,0) -- (1,0);
\draw[-stealth,black,ultra thick] (1,0) -- (2,0);
\draw[-stealth,black,ultra thick] (2,0) -- (2,1);
\draw[-stealth,black,ultra thick] (2,1) -- (2,2);
\draw[-stealth,black,ultra thick] (2,2) -- (2,3);
\draw[-stealth,black,ultra thick] (2,3) -- (3,3);
\end{tikzpicture}
\end{subfigure}
\hfill
\begin{subfigure}{0.18\textwidth}
\centering
\begin{tikzpicture}[scale=0.55]
\draw[step=1cm,gray,very thin] (0,0) grid (3,3);
\draw[gray, dotted, very thick] (0,0) -- (3,3);
\draw[-stealth,black,ultra thick] (0,0) -- (1,0);
\draw[-stealth,black,ultra thick] (1,0) -- (1,1);
\draw[-stealth,black,ultra thick] (1,1) -- (2,1);
\draw[-stealth,black,ultra thick] (2,1) -- (2,2);
\draw[-stealth,black,ultra thick] (2,2) -- (2,3);
\draw[-stealth,black,ultra thick] (2,3) -- (3,3);
\end{tikzpicture}
\end{subfigure}
\hfill
\begin{subfigure}{0.18\textwidth}
\centering
\begin{tikzpicture}[scale=0.55]
\draw[step=1cm,gray,very thin] (0,0) grid (3,3);
\draw[gray, dotted, very thick] (0,0) -- (3,3);
\draw[-stealth,black,ultra thick] (0,0) -- (0,1);
\draw[-stealth,black,ultra thick] (0,1) -- (1,1);
\draw[-stealth,black,ultra thick] (1,1) -- (2,1);
\draw[-stealth,black,ultra thick] (2,1) -- (3,1);
\draw[-stealth,black,ultra thick] (3,1) -- (3,2);
\draw[-stealth,black,ultra thick] (3,2) -- (3,3);
\end{tikzpicture}
\end{subfigure}
\hfill
\begin{subfigure}{0.18\textwidth}
\centering
\begin{tikzpicture}[scale=0.55]
\draw[step=1cm,gray,very thin] (0,0) grid (3,3);
\draw[gray, dotted, very thick] (0,0) -- (3,3);
\draw[-stealth,black,ultra thick] (0,0) -- (1,0);
\draw[-stealth,black,ultra thick] (1,0) -- (1,1);
\draw[-stealth,black,ultra thick] (1,1) -- (1,2);
\draw[-stealth,black,ultra thick] (1,2) -- (2,2);
\draw[-stealth,black,ultra thick] (2,2) -- (3,2);
\draw[-stealth,black,ultra thick] (3,2) -- (3,3);
\end{tikzpicture}
\end{subfigure}
\hfill
\begin{subfigure}{0.18\textwidth}
\centering
\begin{tikzpicture}[scale=0.55]
\draw[step=1cm,gray,very thin] (0,0) grid (3,3);
\draw[gray, dotted, very thick] (0,0) -- (3,3);
\draw[-stealth,black,ultra thick] (0,0) -- (0,1);
\draw[-stealth,black,ultra thick] (0,1) -- (1,1);
\draw[-stealth,black,ultra thick] (1,1) -- (2,1);
\draw[-stealth,black,ultra thick] (2,1) -- (2,2);
\draw[-stealth,black,ultra thick] (2,2) -- (3,2);
\draw[-stealth,black,ultra thick] (3,2) -- (3,3);
\end{tikzpicture}
\end{subfigure}

\vspace{0.5cm}

\centering
\begin{subfigure}{0.18\textwidth}
\centering
\begin{tikzpicture}[scale=0.55]
\draw[step=1cm,gray,very thin] (0,0) grid (3,3);
\draw[gray, dotted, very thick] (0,0) -- (3,3);
\draw[-stealth,black,ultra thick] (0,0) -- (1,0);
\draw[-stealth,black,ultra thick] (1,0) -- (1,1);
\draw[-stealth,black,ultra thick] (1,1) -- (1,2);
\draw[-stealth,black,ultra thick] (1,2) -- (1,3);
\draw[-stealth,black,ultra thick] (1,3) -- (2,3);
\draw[-stealth,black,ultra thick] (2,3) -- (3,3);
\end{tikzpicture}
\end{subfigure}
\hfill
\begin{subfigure}{0.18\textwidth}
\centering
\begin{tikzpicture}[scale=0.55]
\draw[step=1cm,gray,very thin] (0,0) grid (3,3);
\draw[gray, dotted, very thick] (0,0) -- (3,3);
\draw[-stealth,black,ultra thick] (0,0) -- (0,1);
\draw[-stealth,black,ultra thick] (0,1) -- (0,2);
\draw[-stealth,black,ultra thick] (0,2) -- (1,2);
\draw[-stealth,black,ultra thick] (1,2) -- (2,2);
\draw[-stealth,black,ultra thick] (2,2) -- (3,2);
\draw[-stealth,black,ultra thick] (3,2) -- (3,3);
\end{tikzpicture}
\end{subfigure}
\hfill
\begin{subfigure}{0.18\textwidth}
\centering
\begin{tikzpicture}[scale=0.55]
\draw[step=1cm,gray,very thin] (0,0) grid (3,3);
\draw[gray, dotted, very thick] (0,0) -- (3,3);
\draw[-stealth,black,ultra thick] (0,0) -- (1,0);
\draw[-stealth,black,ultra thick] (1,0) -- (1,1);
\draw[-stealth,black,ultra thick] (1,1) -- (1,2);
\draw[-stealth,black,ultra thick] (1,2) -- (2,2);
\draw[-stealth,black,ultra thick] (2,2) -- (2,3);
\draw[-stealth,black,ultra thick] (2,3) -- (3,3);
\end{tikzpicture}
\end{subfigure}
\hfill
\begin{subfigure}{0.18\textwidth}
\centering
\begin{tikzpicture}[scale=0.55]
\draw[step=1cm,gray,very thin] (0,0) grid (3,3);
\draw[gray, dotted, very thick] (0,0) -- (3,3);
\draw[-stealth,black,ultra thick] (0,0) -- (0,1);
\draw[-stealth,black,ultra thick] (0,1) -- (1,1);
\draw[-stealth,black,ultra thick] (1,1) -- (2,1);
\draw[-stealth,black,ultra thick] (2,1) -- (2,2);
\draw[-stealth,black,ultra thick] (2,2) -- (2,3);
\draw[-stealth,black,ultra thick] (2,3) -- (3,3);
\end{tikzpicture}
\end{subfigure}
\hfill
\begin{subfigure}{0.18\textwidth}
\centering
\begin{tikzpicture}[scale=0.55]
\draw[step=1cm,gray,very thin] (0,0) grid (3,3);
\draw[gray, dotted, very thick] (0,0) -- (3,3);
\draw[-stealth,black,ultra thick] (0,0) -- (0,1);
\draw[-stealth,black,ultra thick] (0,1) -- (1,1);
\draw[-stealth,black,ultra thick] (1,1) -- (1,2);
\draw[-stealth,black,ultra thick] (1,2) -- (2,2);
\draw[-stealth,black,ultra thick] (2,2) -- (3,2);
\draw[-stealth,black,ultra thick] (3,2) -- (3,3);
\end{tikzpicture}
\end{subfigure}

\vspace{0.5cm}

\centering
\begin{subfigure}{0.18\textwidth}
\centering
\begin{tikzpicture}[scale=0.55]
\draw[step=1cm,gray,very thin] (0,0) grid (3,3);
\draw[gray, dotted, very thick] (0,0) -- (3,3);
\draw[-stealth,black,ultra thick] (0,0) -- (0,1);
\draw[-stealth,black,ultra thick] (0,1) -- (0,2);
\draw[-stealth,black,ultra thick] (0,2) -- (0,3);
\draw[-stealth,black,ultra thick] (0,3) -- (1,3);
\draw[-stealth,black,ultra thick] (1,3) -- (2,3);
\draw[-stealth,black,ultra thick] (2,3) -- (3,3);
\end{tikzpicture}
\end{subfigure}
\hfill
\begin{subfigure}{0.18\textwidth}
\centering
\begin{tikzpicture}[scale=0.55]
\draw[step=1cm,gray,very thin] (0,0) grid (3,3);
\draw[gray, dotted, very thick] (0,0) -- (3,3);
\draw[-stealth,black,ultra thick] (0,0) -- (0,1);
\draw[-stealth,black,ultra thick] (0,1) -- (1,1);
\draw[-stealth,black,ultra thick] (1,1) -- (1,2);
\draw[-stealth,black,ultra thick] (1,2) -- (1,3);
\draw[-stealth,black,ultra thick] (1,3) -- (2,3);
\draw[-stealth,black,ultra thick] (2,3) -- (3,3);
\end{tikzpicture}
\end{subfigure}
\hfill
\begin{subfigure}{0.18\textwidth}
\centering
\begin{tikzpicture}[scale=0.55]
\draw[step=1cm,gray,very thin] (0,0) grid (3,3);
\draw[gray, dotted, very thick] (0,0) -- (3,3);
\draw[-stealth,black,ultra thick] (0,0) -- (0,1);
\draw[-stealth,black,ultra thick] (0,1) -- (0,2);
\draw[-stealth,black,ultra thick] (0,2) -- (1,2);
\draw[-stealth,black,ultra thick] (1,2) -- (1,3);
\draw[-stealth,black,ultra thick] (1,3) -- (2,3);
\draw[-stealth,black,ultra thick] (2,3) -- (3,3);
\end{tikzpicture}
\end{subfigure}
\hfill
\begin{subfigure}{0.18\textwidth}
\centering
\begin{tikzpicture}[scale=0.55]
\draw[step=1cm,gray,very thin] (0,0) grid (3,3);
\draw[gray, dotted, very thick] (0,0) -- (3,3);
\draw[-stealth,black,ultra thick] (0,0) -- (0,1);
\draw[-stealth,black,ultra thick] (0,1) -- (0,2);
\draw[-stealth,black,ultra thick] (0,2) -- (1,2);
\draw[-stealth,black,ultra thick] (1,2) -- (2,2);
\draw[-stealth,black,ultra thick] (2,2) -- (2,3);
\draw[-stealth,black,ultra thick] (2,3) -- (3,3);
\end{tikzpicture}
\end{subfigure}
\hfill
\begin{subfigure}{0.18\textwidth}
\centering
\begin{tikzpicture}[scale=0.55]
\draw[step=1cm,gray,very thin] (0,0) grid (3,3);
\draw[gray, dotted, very thick] (0,0) -- (3,3);
\draw[-stealth,black,ultra thick] (0,0) -- (0,1);
\draw[-stealth,black,ultra thick] (0,1) -- (1,1);
\draw[-stealth,black,ultra thick] (1,1) -- (1,2);
\draw[-stealth,black,ultra thick] (1,2) -- (2,2);
\draw[-stealth,black,ultra thick] (2,2) -- (2,3);
\draw[-stealth,black,ultra thick] (2,3) -- (3,3);
\end{tikzpicture}
\end{subfigure}

\vspace{0.5cm}

\caption{Partitions induce by Callan's proof.}
\label{callanpartitions}
\end{figure}

\begin{remark}
    In the introduction to Callan's paper, he notes that the proof presented is derived from two papers, one by Blackwell and Hodges, and the other by Narayana. The paper by Narayana contains the words ``cyclic permutations'', although Callan does not make mention of his proof following a cyclic structure.
\end{remark}

There are other examples of single-parameter path-swapping, such as Woan \cite{woan2001uniform}. However, for the variations discussed in this paper, we will describe them in terms of the Callan proof.

\subsection{Multiple parameter path-swapping}

Many multijections which have some sort of path-swapping using involve a single-parameter. However, some path-swapping multijection offer more complex path-swapping, utilizing multiple parameters, and breaking down a path into more pieces than the multijections we discussed previously. To distinguish these, we call these \textit{multiple-step path-swapping}.

As mentioned earlier, our example of multiple-step path-swapping is Young-Ming Chen's proof of the Chung-Feller Theorem \cite{chen2008chung}. Just as we did with Callan's proof, we will briefly explain Chen's technique, and leave the reader to consult Chen's paper for the rigorous details. We also will do Chen's proof as well as its variations strictly in mountain notation, as well in his notation of ``flaws'' as the number of down-steps below the $x$-axis.

First, Chen fixes $n\in\mathbb{N}$, and chooses a lattice path $P$ of length $2n$ with $k$ flaws. He decomposes it into $BuAdC$, where $A,B,C$ are subpaths, and $u$ and $d$ denote the first up-step above the $x$-axis and first down-step touching the $x$-axis after $u$. Then, he reassembles a new path as $AdBuC$, proving this path has $k+1$ flaws. In the paper, Chen proves that this permutation of paths is reversible, thus proving his bijection. Chen also shows the partitions induced by his bijections in his paper for $n=3$.

\begin{figure}
\begin{center}
\begin{tikzpicture}
\draw[step=1cm,gray,very thin] (0,-2) grid (12,2);
\draw[gray, dotted, very thick] (0,0) -- (8,0);
\draw[-stealth,blue,ultra thick] (0,0) -- (1,-1);
\draw[-stealth,blue,ultra thick] (1,-1) -- (2,0);
\draw[-stealth,magenta,ultra thick] (2,0) -- (3,1);
\draw[-stealth,teal,ultra thick] (3,1) -- (4,2);
\draw[-stealth,teal,ultra thick] (4,2) -- (5,1);
\draw[-stealth,magenta,ultra thick] (5,1) -- (6,0);
\draw[-stealth,olive,ultra thick] (6,0) -- (7,-1);
\draw[-stealth,olive,ultra thick] (7,-1) -- (8,0);
\draw[-stealth,olive,ultra thick] (8,0) -- (9,-1);
\draw[-stealth,olive,ultra thick] (9,-1) -- (10,-2);
\draw[-stealth,olive,ultra thick] (10,-2) -- (11,-1);
\draw[-stealth,olive,ultra thick] (11,-1) -- (12,0);
\end{tikzpicture}
\end{center}
\caption{Chen's subpaths on a path between $(0,0)$ and $(12,0)$, \textcolor{teal}{A}, \textcolor{blue}{B}, and \textcolor{olive}{C}.}
\label{chen}
\end{figure}
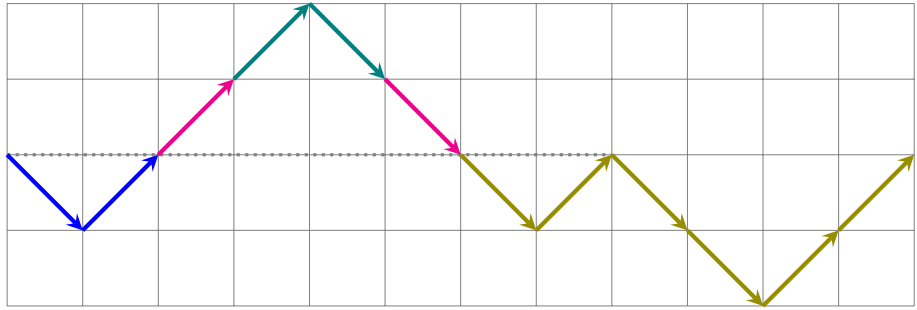

As one can see, Chen's proof relies on the identification of two steps on the lattice path, one up-step and one down-step. This is different compared to Callan's reliance on a single up-step in order to establish the bijection between $S_k$ and $S_{k+1}$. Other proofs with a multiple-step format include Rubenstein\cite{rubenstein1994catalan}. The reader should notice these proofs tend to be more complicated in their methods, although still producing a beautiful bijection. With both proofs described, we move to discuss variations on the bijections given by Callan and Chen.

\section{The mirror}

In the last section, we asked the question of when some multijections procedure the same partitions. Path-swapping proofs which offer bijections such as Callan's \textit{Pair Them Up!} \cite{callan1995pair} mimic a cyclic structure. However, we would like to find new multijections which do not carry the inherent cycle structure of the Cycle Lemma, as well as many path-swapping proofs. This motivates our definition of the mirror. For defining the mirror and related propositions, we will use staircase notation.

\begin{definition}[The opposite]
    Let $P$ be a lattice path, and $p_i$ be the $i$-th step in the path. Let $\overline{p_i}$ denote the \textit{opposite} step; if $P$ is defined on the basis of up-steps and down-steps and $p_i$ is a up-step, then $\overline{p_i}$ is a down-step.
\end{definition}

The definition above is well-defined since lattice paths consisting of up-steps and down-steps or up-steps and right-steps are themselves defined on a binary alphabet. Naturally, this definition can be extended to other binary alphabets. This gives us a more rigorous footing to define the mirror of a lattice path.

\begin{definition}[Mirror]
    Let $P$ be a lattice path, and its representation be $p_1\dots p_n$. Then the mirror of $P$, which we denote with $\overline{P}$ is $\overline{p_n}\,\overline{p_{n-1}}\dots\overline{p_{2}}\,\overline{p_{1}}$.
\end{definition}

\begin{figure}
\begin{subfigure}{0.5\textwidth}
\begin{tikzpicture}
\draw[step=1cm,gray,very thin] (0,-2) grid (6,2);
\draw[gray, dotted, very thick] (0,0) -- (6,0);
\draw[-stealth,magenta,ultra thick] (0,0) -- (1,1);
\draw[-stealth,magenta,ultra thick] (1,1) -- (2,2);
\draw[-stealth,magenta,ultra thick] (2,2) -- (3,1);
\draw[-stealth,magenta,ultra thick] (3,1) -- (4,0);
\draw[-stealth,magenta,ultra thick] (4,0) -- (5,1);
\draw[-stealth,magenta,ultra thick] (5,1) -- (6,0);
\end{tikzpicture}
\caption{A lattice path $A$.}
\end{subfigure}
\begin{subfigure}{0.5\textwidth}
\begin{tikzpicture}
\draw[step=1cm,gray,very thin] (0,-2) grid (6,2);
\draw[gray, dotted, very thick] (0,0) -- (6,0);
\draw[-stealth,magenta,ultra thick] (0,0) -- (1,1);
\draw[-stealth,magenta,ultra thick] (1,1) -- (2,0);
\draw[-stealth,magenta,ultra thick] (2,0) -- (3,1);
\draw[-stealth,magenta,ultra thick] (3,1) -- (4,2);
\draw[-stealth,magenta,ultra thick] (4,2) -- (5,1);
\draw[-stealth,magenta,ultra thick] (5,1) -- (6,0);
\end{tikzpicture}
\caption{$\overline{A}$.}
\end{subfigure}
\caption{A lattice path $P$, and its mirror $\overline{P}$.}
\label{mirror}
\end{figure}

Intuitively, the mirror of a lattice is the path that would run if you held a mirror up to that lattice path. The rigorous definition makes it easier to prove a few propositions and properties of this. First, revisiting the opposite operation on a up-step/right-step, we would expect that if the opposite of a up-step is a right-step, then the opposite of the opposite of an up-step is an up-step, and similarly.

\begin{proposition}\label{steporder2}
    Let $p$ be an up-step or right-step in a lattice path. Then, $\overline{\overline{p}}\cong v$.
\end{proposition}

\begin{proof}
    This comes from the fact that for our lattice paths, we only consider two types of steps; a binary alphabet if you will. We could do casework on whether $v$ is an up-step or right-step.

    If $p$ is an up-step, then $\overline{p}$ is a right-step. This would imply that $\overline{\overline{p}}$ is an up-step. If $v$ is a right-step, then $\overline{p}$ is an up-step, and $\overline{\overline{p}}=p$.
    
\end{proof}

\begin{proposition}\label{pathorder2}
    Let $P$ be a lattice path. Then $\overline{\overline{P}}=P$.
\end{proposition}

\begin{proof}
    Suppose that $P=p_1p_2\dots p_{2n}$. Then $\overline{P}=\overline{p_{2n}}\,\overline{p_{2n-1}}\dots \overline{p_1}$. This implies that $\overline{\overline{P}}=\overline{\overline{p_1}}\,\overline{\overline{p_1}}\dots\overline{\overline{p_{2n}}}$. By \ref{steporder2}, we can rewrite this as $p_1p_2\dots p_{2n}$, so we are done.
\end{proof}

\begin{proposition}\label{uniqueness}
    If $P$ is a lattice path, then $\overline{P}$ is unique, i.e. if there exists two lattice paths $P_1,P_2$ such that $\overline{P_1}\cong \overline{P_2}$, then $P_1\cong P_2$.
\end{proposition}

\begin{proof}
    Let $P_1=p_1p_2\dots p_n$ and $P_2=q_1q_2\dots q_n$. By our definition of the mirror, then $\overline{P_1}=\overline{p_n}\,\overline{p_{n-1}}\dots\overline{p_1}$ and $\overline{P_2}=\overline{q_n}\,\overline{q_{n-1}}\dots\overline{q_1}$. If $\overline{P_1}\cong \overline{P_2}$, then $\overline{p_i}\cong\overline{q_i}$ for $i=1,\dots,n$. Since our lattice path is defined on a binary alphabet, then $p_i\cong q_i$ for $i=1,\dots,n$. Additionally, if $\overline{P_1}=\overline{P_2}$, then $P_1$ and $P_2$ must have the same length. By our definition in \ref{latticepathequivalence}, then $P_1\cong P_2$.
\end{proof}

The last proposition is important, since it establishes the mirror as a unique operation. We wish to apply this to constructing bijections between sets of lattice paths. We resume with the next proposition.

\begin{proposition}\label{sameupdown}
    Let $P$ be a lattice path in staircase notation. Then $P$ and $\overline{P}$ have the same amount of up-steps and right-steps above the diagonal, and the same amount of up-steps and right-steps below the diagonal.
\end{proposition}

\begin{proof}
    First, we prove this for lattice paths which only touch the diagonal at the start and finish. Then, this will be true for all lattice paths, since any lattice path will be a concatenation of lattice paths which only touch the diagonal at their ends.

    We will use induction on $n$ to prove this for lattice paths of length $2n$ which only touch the diagonal at start and finish. For $n=1$, there are two lattice paths for which we can compute the number of up-steps and right-steps above and below the diagonal respectively. Our proposition is clearly true, as shown in Figure \ref{basecase}.

    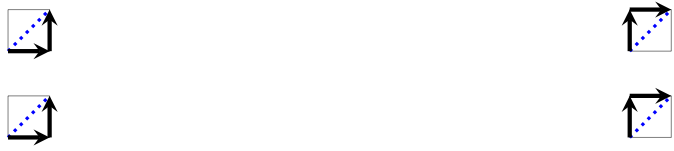
\begin{figure}
        \centering
        \begin{subfigure}{0.4\textwidth}
        \centering
        \begin{tikzpicture}[scale=0.55]
        \draw[step=1cm,gray,very thin] (0,0) grid (1,1);
        \draw[blue, dotted, very thick] (0,0) -- (1,1);
        \draw[-stealth,black,ultra thick] (0,0) -- (1,0);
        \draw[-stealth,black,ultra thick] (1,0) -- (1,1);
        \end{tikzpicture}
        \end{subfigure}
        \hfill
        \begin{subfigure}{0.4\textwidth}
        \centering
        \begin{tikzpicture}[scale=0.55]
        \draw[step=1cm,gray,very thin] (0,0) grid (1,1);
        \draw[blue, dotted, very thick] (0,0) -- (1,1);
        \draw[-stealth,black,ultra thick] (0,0) -- (0,1);
        \draw[-stealth,black,ultra thick] (0,1) -- (1,1);
        \end{tikzpicture}
        \end{subfigure}
        
        \vspace{0.5cm}

        \begin{subfigure}{0.4\textwidth}
        \centering
        \begin{tikzpicture}[scale=0.55]
        \draw[step=1cm,gray,very thin] (0,0) grid (1,1);
        \draw[blue, dotted, very thick] (0,0) -- (1,1);
        \draw[-stealth,black,ultra thick] (0,0) -- (1,0);
        \draw[-stealth,black,ultra thick] (1,0) -- (1,1);
        \end{tikzpicture}
        \end{subfigure}
        \hfill
        \begin{subfigure}{0.4\textwidth}
        \centering
        \begin{tikzpicture}[scale=0.55]
        \draw[step=1cm,gray,very thin] (0,0) grid (1,1);
        \draw[blue, dotted, very thick] (0,0) -- (1,1);
        \draw[-stealth,black,ultra thick] (0,0) -- (0,1);
        \draw[-stealth,black,ultra thick] (0,1) -- (1,1);
        \end{tikzpicture}
        \end{subfigure}
        \caption{Lattice path with $1$ up-steps and $1$ right-step. The respective mirrors share the same up-steps and right-steps above and below the diagonal.}
        \label{basecase}
    \end{figure}

    Now, we assume this is true for lattice paths of this condition with at most $n$ up-steps and $n$ right-steps only touching the diagonal at ends. Now, we want to show this is true for lattice path containing $n+1$ up-steps, $n+1$ right-steps, and only touches the diagonal at the ends of the path. Let $P$ be a lattice path with this condition, and let $\overline{P}$ be its mirror. Without loss of generality, we assume $P$ lies below the diagonal, so it has $n+1$ up-steps and $n+1$ right-step below the diagonal, and zero of either step above the diagonal. For contradiction, we assume the mirror does not lie entirely below the diagonal. This implies that $\overline{P}$ intersects the diagonal at $(k_1,k_1)$ for some $k_1\in\mathbb{N}$, $k_1<n$. We will assume this is the first intersection at the diagonal by following the path from the start. Denote this subpath as $\overline{P_{k_1}}$. We can continue to do this by following the path along to the next time it touches the diagonal. This will eventually give us subpaths $P_{k_2},P_{k_3},\dots,P_{k_m}$, where $k_m$ is simply the final segment of the path that ends at $(n,n)$. Each $P_{k_m}$ is a path which only touches the diagonal at $(k_{m-1},k_{m-1})$ and $(k_m,k_m)$. So for $\overline{P_{k_m}}$, by construction, then $(k_m,k_m)=(n,n)$. Similarly, since we have that each segment either lies above or below the diagonal.

    For $\overline{P_{k_i}}$, if it starts at $(k_{i-1},k_{i-1})$ and ends at $(k_{i},k_i)$, the path either has $k_i-k_{i-1}$ up-steps below the diagonal and $k_i-k_{i-1}$ right-steps below the diagonal, or the same up-steps and right-steps above the diagonal. We also have by construction that $k_1<k_2<\dots<k_n=n$, so $k_i-k_{i+1}\geq 0$.

    In total, $\overline{P}=P_{k_1}P_{k_2}\dots P_{k_m}$. By \ref{pathorder2}, then $\overline{\overline{P}}=P$. So we have that $P=\overline{P_{k_m}}\,\overline{P_{m-1}}\dots \overline{P_2}\,\overline{P_1}$. Suppose $P_{k_i}$ lies entirely below the diagonal. Since $P_{k_i}$ has length less than $n+1$, by the induction hypothesis, then the mirror of $P_{k_i}$ also lies entirely below the diagonal. We have a similar case if $P_{k_i}$ lies completely below the diagonal. If $\overline{P}$ does not lie entirely below the diagonal, at least one $k_i$ segment lies entirely above the diagonal, call it $P_{k_\ell}$ This segment includes at least one up-step and at least one down-step above the diagonal, and by induction hypothesis, this is also true of $\overline{P_{k_\ell}}$ So $P$ must contain at least one up-step and one down-step above the diagonal, and at maximum $n$ up-steps and $n$ right-steps below the diagonal.

    However, this is a contradiction as we assumed $P$ to lie below the diagonal completely, which means that all $n+1$ up-steps and $n+1$ right-steps lied below the diagonal. Hence, this is true for all $n\in\mathbb{N}$, and for all lattice paths which touch the diagonal at the ends.

    Now, let $P$ be any lattice path in staircase notation,so it starts at $(0,0)$, ends at $(n,n)$ for $n\in\mathbb{N}$, and consists of up-steps and right-steps. Follow $P$ until it first touches the diagonal, at some $(k_1,k_1)$ for $k_1\leq n$. The path segment from $(0,0)$ to $(k_1,k_1)$ will be denoted by $P_{k_1}$. Continue this for each time $P$ touches the diagonal until $P=P_{k_1}P_{k_2}\dots P_{k_m}$. By construction, then each $P_{k_i}$ is a subpath which touches the diagonal at the ends. Let $c_{above}$ denote the set of subpaths $P_{k_i}$ such that $P_{k_i}$ lies strictly above the diagonal, and $c_{below}$ denote the subpaths $P_{k_i}$ that lie strictly below the diagonal. For $c_{\text{above}}$, associate a $x_{\text{above}}$ that tallies the total up-steps and right-steps above the diagonal, and do the same to $c_{\text{below}}$ with $x_{\text{below}}$. By the earlier work, for $P_{k_i}\in c_{\text{above}}$, $\overline{P_{k_i}}$ must also contains the same amount of up-steps and right-steps above the diagonal, and zero below the diagonal. Similarly, this extends to $c_{\text{below}}$. By the definition of the mirror, $\overline{P}=\overline{P_{k_m}}\,\overline{P_{k_{m-1}}}\dots \overline{P_{k_1}}$. Again, denote the sets $\overline{c}_{\text{above}}$ and $\overline{c}_{\text{below}}$ be the sets of subpaths that lie strictly above and below the diagonal respectively. By the above, $c_{\text{above}}$ $\overline{c}_{\text{above}}$ contains the same root paths, i.e $P_{k_i}\in c_{\text{above}}$ implies $\overline{P_{k_i}}\in \overline{c}_{\text{above}}$. Similarly, this follows for $c_{\text{below}}$ and $\overline{c}_\text{below}$. If we associate $\overline{x}_{\text{above}}$ to $\overline{c}_{above}$ in the same above fashion, then since $c_{\text{above}}$ and $\overline{c}_{\text{above}}$ have the same root set, then $x_{\text{above}}=\overline{x}_{\text{above}}$. We have the same setup for $\overline{x}_{\text{below}}$ to $\overline{c}_{\text{below}}$.

    If $x_{\text{above}}$ and $\overline{x}_{\text{above}}$ are the same, and $x_{\text{below}}$ and $\overline{x}_{\text{below}}$ are the same, then this proves that $P$ and $\overline{P}$ have the same up-steps and right-steps above and below the diagonal.
    
\end{proof}

We note this proof is easily generalizable to the mountain notation case. One key thing to notice that in terms of flaws, the last proposition shows that the mirror of a lattice path with $k$ flaws also has $k$ flaws.

\begin{proposition}\label{mubijection}
    Assume the setup of previous multijections, let $n\in\mathbb{N}$, and consider the set of lattice paths in staircase notation from $(0,0)$ to $(n,n)$. Divide this into sets $S_k$, for $k=0,1,\dots, n$, where $S_k$ contains lattice paths with $k$ flaws. Denote the function $\mu:S_k\to S_k$ for $P\in S_k$ by $\pi(P)=\overline{P}$. Then $\mu$ is a bijection.
\end{proposition}

\begin{proof}
    Notice that $\mu$ is well-defined since the mirror operation is well-defined. We need to prove injectivity and surjectivity. For surjectivity, let $P$ be a lattice path in $S_k$. Let $Q=\overline{P}$. By \ref{pathorder2}, then $\mu(Q)=\overline{Q}=P$. By \ref{sameupdown}, $Q\in S_k$. This is true for all paths in $S_k$, so $\mu$ is surjective. For injectivity, suppose there exists $P_1,P_2\in S_k$ such that $\mu(P_1)=\mu(P_2)$. Then $\overline{P_1}=\overline{P_2}$, and by \ref{uniqueness}, then $P_1=P_2$. So $\mu$ is injective. Hence, $\mu$ is a bijection from $S_k$ to itself. Again, this proof is similar when a lattice path is in mountain notation.
\end{proof}

Finally, we refer to lattice paths for which $P=\overline{P}$ as symmetric lattice paths. Two examples of short symmetric paths are given in Figure \ref{basecase}. We conclude our study of properties of the mirror with the following remark. 

\begin{remark}
    For smaller values of $n\in\mathbb{N}$, the mirror may not be as thought-provoking of a concept. Consider that for $n=1$, so a lattice path of length $2n$, the mirror of any lattice path of this length is itself. However, as $n$ grows, the possible lattice paths grows, and ``symmetric'' lattice paths become less common.
\end{remark}

The last few propositions will give us enough to finally define mirror variations of exisiting multijections. In these cases, we will define the variations through bijections between $S_k$ and $S_{k+1}$ in the context of the Callan proof and the Chen proof. Then, we will show choose one Dyck path, and compare its preimages under both the original multijection and the mirror variation.

\subsection{A mirror variation of Callan's proof}

Callan's proof follows what appears to be a cyclic structure, since we swap around one single chosen point. We assume the setup of Callan's paper.

\begin{theorem}[Mirror Variation of Callan's Proof]\label{callanmirror}
    There is a bijection between sets $S_k$ and $S_{k+1}$.
\end{theorem}

\begin{proof}
    Denote the bijection Callan constructed between $S_k$ and $S_{k+1}$ as $\phi_k$, so that $\phi_k:S_k\to S_{k+1}$. For any $k$, by \ref{mubijection}, then $\mu$ is a bijection from $S_k$ to itself. Define $\Phi_k:S_k\to S_{k+1}$ as the composition $\mu\circ \phi_k$. The composition of two bijections is again a bijection, so $\Phi_k$ is a valid bijection between $S_k$ and $S_{k+1}$ with $\Phi_k^{-1}:S_{k+1}\to S_k$. 

\end{proof}

With this bijection, what does the resulting multijection look like? We again describe the multijection by its relative partitions. If $D$ is a Dyck path, then the preimage of $D$ under the multijection would be the set
\begin{equation*}
    \{D, \Phi_1(D),\Phi_2(\Phi_1(D)),\dots, \Phi_n(\Phi_{n-1}(\dots(\Phi_1(D))\}
\end{equation*}

We describe the resulting partitions in similar format to the early figure in \ref{callanpartitions}. We encourage the reader to notice how the partitions are different.

\begin{figure}
        \centering
        \begin{subfigure}{0.18\textwidth}
        \centering
        \begin{tikzpicture}[scale=0.55]
            \draw[step=1cm,gray,very thin] (0,0) grid (3,3);
            \draw[blue, dotted, very thick] (0,0) -- (3,3);
            \draw[-stealth,black,ultra thick] (0,0) -- (1,0);
            \draw[-stealth,black,ultra thick] (1,0) -- (2,0);
            \draw[-stealth,black,ultra thick] (2,0) -- (3,0);
            \draw[-stealth,black,ultra thick] (3,0) -- (3,1);
            \draw[-stealth,black,ultra thick] (3,1) -- (3,2);
            \draw[-stealth,black,ultra thick] (3,2) -- (3,3);
        \end{tikzpicture}
        \end{subfigure}
        \hfill
        \begin{subfigure}{0.18\textwidth}
        \centering
        \begin{tikzpicture}[scale=0.55]
            \draw[step=1cm,gray,very thin] (0,0) grid (3,3);
            \draw[blue, dotted, very thick] (0,0) -- (3,3);
            \draw[-stealth,black,ultra thick] (0,0) -- (1,0);
            \draw[-stealth,black,ultra thick] (1,0) -- (2,0);
            \draw[-stealth,black,ultra thick] (2,0) -- (2,1);
            \draw[-stealth,black,ultra thick] (2,1) -- (3,1);
            \draw[-stealth,black,ultra thick] (3,1) -- (3,2);
            \draw[-stealth,black,ultra thick] (3,2) -- (3,3);
        \end{tikzpicture}
        \end{subfigure}
        \hfill
        \begin{subfigure}{0.18\textwidth}
        \centering
        \begin{tikzpicture}[scale=0.55]
            \draw[step=1cm,gray,very thin] (0,0) grid (3,3);
            \draw[blue, dotted, very thick] (0,0) -- (3,3);
            \draw[-stealth,black,ultra thick] (0,0) -- (1,0);
            \draw[-stealth,black,ultra thick] (1,0) -- (2,0);
            \draw[-stealth,black,ultra thick] (2,0) -- (2,1);
            \draw[-stealth,black,ultra thick] (2,1) -- (2,2);
            \draw[-stealth,black,ultra thick] (2,2) -- (3,2);
            \draw[-stealth,black,ultra thick] (3,2) -- (3,3);
        \end{tikzpicture}
        \end{subfigure}
        \begin{subfigure}{0.18\textwidth}
        \centering
        
        \begin{tikzpicture}[scale=0.55]
            \draw[step=1cm,gray,very thin] (0,0) grid (3,3);
            \draw[blue, dotted, very thick] (0,0) -- (3,3);
            \draw[-stealth,black,ultra thick] (0,0) -- (1,0);
            \draw[-stealth,black,ultra thick] (1,0) -- (1,1);
            \draw[-stealth,black,ultra thick] (1,1) -- (2,1);
            \draw[-stealth,black,ultra thick] (2,1) -- (3,1);
            \draw[-stealth,black,ultra thick] (3,1) -- (3,2);
            \draw[-stealth,black,ultra thick] (3,2) -- (3,3);
        \end{tikzpicture}
        \end{subfigure}
        \hfill
        \begin{subfigure}{0.18\textwidth}
        \centering
        \begin{tikzpicture}[scale=0.55]
            \draw[step=1cm,gray,very thin] (0,0) grid (3,3);
            \draw[blue, dotted, very thick] (0,0) -- (3,3);
            \draw[-stealth,black,ultra thick] (0,0) -- (1,0);
            \draw[-stealth,black,ultra thick] (1,0) -- (1,1);
            \draw[-stealth,black,ultra thick] (1,1) -- (2,1);
            \draw[-stealth,black,ultra thick] (2,1) -- (2,2);
            \draw[-stealth,black,ultra thick] (2,2) -- (3,2);
            \draw[-stealth,black,ultra thick] (3,2) -- (3,3);
        \end{tikzpicture}
        \end{subfigure}
        
        \vspace{0.5cm}

        \begin{subfigure}{0.18\textwidth}
        \centering
        \begin{tikzpicture}[scale=0.55]
            \draw[step=1cm,gray,very thin] (0,0) grid (3,3);
            \draw[blue, dotted, very thick] (0,0) -- (3,3);
            \draw[-stealth,black,ultra thick] (0,0) -- (0,1);
            \draw[-stealth,black,ultra thick] (0,1) -- (1,1);
            \draw[-stealth,black,ultra thick] (1,1) -- (2,1);
            \draw[-stealth,black,ultra thick] (2,1) -- (3,1);
            \draw[-stealth,black,ultra thick] (3,1) -- (3,2);
            \draw[-stealth,black,ultra thick] (3,2) -- (3,3);
        \end{tikzpicture}
        \end{subfigure}
        \hfill
        \begin{subfigure}{0.18\textwidth}
        \centering
        \begin{tikzpicture}[scale=0.55]
            \draw[step=1cm,gray,very thin] (0,0) grid (3,3);
            \draw[blue, dotted, very thick] (0,0) -- (3,3);
            \draw[-stealth,black,ultra thick] (0,0) -- (0,1);
            \draw[-stealth,black,ultra thick] (0,1) -- (1,1);
            \draw[-stealth,black,ultra thick] (1,1) -- (2,1);
            \draw[-stealth,black,ultra thick] (2,1) -- (2,2);
            \draw[-stealth,black,ultra thick] (2,2) -- (3,2);
            \draw[-stealth, black, ultra thick] (3,2) -- (3,3);
        \end{tikzpicture}
        \end{subfigure}
        \hfill
        \begin{subfigure}{0.18\textwidth}
        \centering
        \begin{tikzpicture}[scale=0.55]
            \draw[step=1cm,gray,very thin] (0,0) grid (3,3);
            \draw[blue, dotted, very thick] (0,0) -- (3,3);
            \draw[-stealth,black,ultra thick] (0,0) -- (1,0);
            \draw[-stealth,black,ultra thick] (1,0) -- (2,0);
            \draw[-stealth,black,ultra thick] (2,0) -- (2,1);
            \draw[-stealth,black,ultra thick] (2,1) -- (2,2);
            \draw[-stealth,black,ultra thick] (2,2) -- (2,3);
            \draw[-stealth, black, ultra thick] (2,3) -- (3,3);
        \end{tikzpicture}
        \end{subfigure}
        \begin{subfigure}{0.18\textwidth}
        \centering
        
        \begin{tikzpicture}[scale=0.55]
            \draw[step=1cm,gray,very thin] (0,0) grid (3,3);
            \draw[blue, dotted, very thick] (0,0) -- (3,3);
            \draw[-stealth,black,ultra thick] (0,0) -- (1,0);
            \draw[-stealth,black,ultra thick] (1,0) -- (1,1);
            \draw[-stealth,black,ultra thick] (1,1) -- (1,2);
            \draw[-stealth,black,ultra thick] (1,2) -- (2,2);
            \draw[-stealth,black,ultra thick] (2,2) -- (3,2);
            \draw[-stealth,black,ultra thick] (3,2) -- (3,3);
        \end{tikzpicture}
        \end{subfigure}
        \hfill
        \begin{subfigure}{0.18\textwidth}
        \centering
        \begin{tikzpicture}[scale=0.55]
            \draw[step=1cm,gray,very thin] (0,0) grid (3,3);
            \draw[blue, dotted, very thick] (0,0) -- (3,3);
            \draw[-stealth,black,ultra thick] (0,0) -- (1,0);
            \draw[-stealth,black,ultra thick] (1,0) -- (1,1);
            \draw[-stealth,black,ultra thick] (1,1) -- (2,1);
            \draw[-stealth,black,ultra thick] (2,1) -- (2,2);
            \draw[-stealth,black,ultra thick] (2,2) -- (2,3);
            \draw[-stealth,black,ultra thick] (2,3) -- (3,3);
        \end{tikzpicture}
        \end{subfigure}

        \vspace{0.5cm}

        \begin{subfigure}{0.18\textwidth}
        \centering
        \begin{tikzpicture}[scale=0.55]
            \draw[step=1cm,gray,very thin] (0,0) grid (3,3);
            \draw[blue, dotted, very thick] (0,0) -- (3,3);
            \draw[-stealth,black,ultra thick] (0,0) -- (0,1);
            \draw[-stealth,black,ultra thick] (0,1) -- (1,1);
            \draw[-stealth,black,ultra thick] (1,1) -- (1,2);
            \draw[-stealth,black,ultra thick] (1,2) -- (2,2);
            \draw[-stealth,black,ultra thick] (2,2) -- (3,2);
            \draw[-stealth,black,ultra thick] (3,2) -- (3,3);
        \end{tikzpicture}
        \end{subfigure}
        \hfill
        \begin{subfigure}{0.18\textwidth}
        \centering
        \begin{tikzpicture}[scale=0.55]
            \draw[step=1cm,gray,very thin] (0,0) grid (3,3);
            \draw[blue, dotted, very thick] (0,0) -- (3,3);
            \draw[-stealth,black,ultra thick] (0,0) -- (1,0);
            \draw[-stealth,black,ultra thick] (1,0) -- (1,1);
            \draw[-stealth,black,ultra thick] (1,1) -- (1,2);
            \draw[-stealth,black,ultra thick] (1,2) -- (2,2);
            \draw[-stealth,black,ultra thick] (2,2) -- (2,3);
            \draw[-stealth,black,ultra thick] (2,3) -- (3,3);
        \end{tikzpicture}
        \end{subfigure}
        \hfill
        \begin{subfigure}{0.18\textwidth}
        \centering
        \begin{tikzpicture}[scale=0.55]
            \draw[step=1cm,gray,very thin] (0,0) grid (3,3);
            \draw[blue, dotted, very thick] (0,0) -- (3,3);
            \draw[-stealth,black,ultra thick] (0,0) -- (0,1);
            \draw[-stealth,black,ultra thick] (0,1) -- (0,2);
            \draw[-stealth,black,ultra thick] (0,2) -- (1,2);
            \draw[-stealth,black,ultra thick] (1,2) -- (2,2);
            \draw[-stealth,black,ultra thick] (2,2) -- (3,2);
            \draw[-stealth,black,ultra thick] (3,2) -- (3,3);
        \end{tikzpicture}
        \end{subfigure}
        \begin{subfigure}{0.18\textwidth}
        \centering
        
        \begin{tikzpicture}[scale=0.55]
            \draw[step=1cm,gray,very thin] (0,0) grid (3,3);
            \draw[blue, dotted, very thick] (0,0) -- (3,3);
            \draw[-stealth,black,ultra thick] (0,0) -- (0,1);
            \draw[-stealth,black,ultra thick] (0,1) -- (1,1);
            \draw[-stealth,black,ultra thick] (1,1) -- (2,1);
            \draw[-stealth,black,ultra thick] (2,1) -- (2,2);
            \draw[-stealth,black,ultra thick] (2,2) -- (2,3);
            \draw[-stealth,black,ultra thick] (2,3) -- (3,3);
        \end{tikzpicture}
        \end{subfigure}
        \hfill
        \begin{subfigure}{0.18\textwidth}
        \centering
        \begin{tikzpicture}[scale=0.55]
            \draw[step=1cm,gray,very thin] (0,0) grid (3,3);
            \draw[blue, dotted, very thick] (0,0) -- (3,3);
            \draw[-stealth,black,ultra thick] (0,0) -- (1,0);
            \draw[-stealth,black,ultra thick] (1,0) -- (1,1);
            \draw[-stealth,black,ultra thick] (1,1) -- (1,2);
            \draw[-stealth,black,ultra thick] (1,2) -- (1,3);
            \draw[-stealth,black,ultra thick] (1,3) -- (2,3);
            \draw[-stealth,black,ultra thick] (2,3) -- (3,3);
        \end{tikzpicture}
        \end{subfigure}

        \vspace{0.5cm}

        \begin{subfigure}{0.18\textwidth}
        \centering
        \begin{tikzpicture}[scale=0.55]
            \draw[step=1cm,gray,very thin] (0,0) grid (3,3);
            \draw[blue, dotted, very thick] (0,0) -- (3,3);
            \draw[-stealth,black,ultra thick] (0,0) -- (0,1);
            \draw[-stealth,black,ultra thick] (0,1) -- (1,1);
            \draw[-stealth,black,ultra thick] (1,1) -- (1,2);
            \draw[-stealth,black,ultra thick] (1,2) -- (2,2);
            \draw[-stealth,black,ultra thick] (2,2) -- (2,3);
            \draw[-stealth,black,ultra thick] (2,3) -- (3,3);
        \end{tikzpicture}
        \end{subfigure}
        \hfill
        \begin{subfigure}{0.18\textwidth}
        \centering
        \begin{tikzpicture}[scale=0.55]
            \draw[step=1cm,gray,very thin] (0,0) grid (3,3);
            \draw[blue, dotted, very thick] (0,0) -- (3,3);
            \draw[-stealth,black,ultra thick] (0,0) -- (0,1);
            \draw[-stealth,black,ultra thick] (0,1) -- (0,2);
            \draw[-stealth,black,ultra thick] (0,2) -- (1,2);
            \draw[-stealth,black,ultra thick] (1,2) -- (1,3);
            \draw[-stealth,black,ultra thick] (1,3) -- (2,3);
            \draw[-stealth,black,ultra thick] (2,3) -- (3,3);
        \end{tikzpicture}
        \end{subfigure}
        \hfill
        \begin{subfigure}{0.18\textwidth}
        \centering
        \begin{tikzpicture}[scale=0.55]
            \draw[step=1cm,gray,very thin] (0,0) grid (3,3);
            \draw[blue, dotted, very thick] (0,0) -- (3,3);
            \draw[-stealth,black,ultra thick] (0,0) -- (0,1);
            \draw[-stealth,black,ultra thick] (0,1) -- (0,2);
            \draw[-stealth,black,ultra thick] (0,2) -- (1,2);
            \draw[-stealth,black,ultra thick] (1,2) -- (2,2);
            \draw[-stealth,black,ultra thick] (2,2) -- (2,3);
            \draw[-stealth,black,ultra thick] (2,3) -- (3,3);
        \end{tikzpicture}
        \end{subfigure}
        \begin{subfigure}{0.18\textwidth}
        \centering
        
        \begin{tikzpicture}[scale=0.55]
            \draw[step=1cm,gray,very thin] (0,0) grid (3,3);
            \draw[blue, dotted, very thick] (0,0) -- (3,3);
            \draw[-stealth,black,ultra thick] (0,0) -- (0,1);
            \draw[-stealth,black,ultra thick] (0,1) -- (1,1);
            \draw[-stealth,black,ultra thick] (1,1) -- (1,2);
            \draw[-stealth,black,ultra thick] (1,2) -- (1,3);
            \draw[-stealth,black,ultra thick] (1,3) -- (2,3);
            \draw[-stealth,black,ultra thick] (2,3) -- (3,3);
        \end{tikzpicture}
        \end{subfigure}
        \hfill
        \begin{subfigure}{0.18\textwidth}
        \centering
        \begin{tikzpicture}[scale=0.55]
            \draw[step=1cm,gray,very thin] (0,0) grid (3,3);
            \draw[blue, dotted, very thick] (0,0) -- (3,3);
            \draw[-stealth,black,ultra thick] (0,0) -- (0,1);
            \draw[-stealth,black,ultra thick] (0,1) -- (0,2);
            \draw[-stealth,black,ultra thick] (0,2) -- (0,3);
            \draw[-stealth,black,ultra thick] (0,3) -- (1,3);
            \draw[-stealth,black,ultra thick] (1,3) -- (2,3);
            \draw[-stealth,black,ultra thick] (2,3) -- (3,3);
        \end{tikzpicture}
        \end{subfigure}

        \caption{Set of partitions when $n=3$ with the simple mirror variation of Callan's proof. Compare to Figure \ref{callanpartitions}.}
        \label{callansimplemirror}
    \end{figure}

\subsection{A mirror variation of Chen's proof}

We have just shown a mirror variation of Callan's proof, and we wish to display two variations of the Chen proof. The first is a simple variation identical to the process we described for a mirror variation of Callan's proof. The second variation is more complex, proposing three different multijections.

\begin{theorem}[Simple Mirror Variation of Chen's Proof]\label{chenmirror}
    There is a bijection between the sets $S_k$ and $S_{k+1}$.
\end{theorem}

\begin{proof}
    Similar to the Callan mirror variation, we propose the most simple mirror variation. Let $\phi_k:S_k\to S_{k+1}$ be Chen's bijection between $S_k$ and $S_{k+1}$ from his original paper, and define $\Phi_k:S_k\to S_{k+1}$ as the composition $\mu\circ \phi_k$. By \ref{mubijection}, $\mu$ is a bijection from $S_{k+1}$ to itself, and the composition of bijections is once again a bijection.
\end{proof}

In Figure \ref{chensimplemirror}, we depict the partitions induced by the simple mirror variation. We encourage the reader to compare them with the the partitions under Chen's original proof, in Figure \ref{chenpartitions}.

\begin{figure}
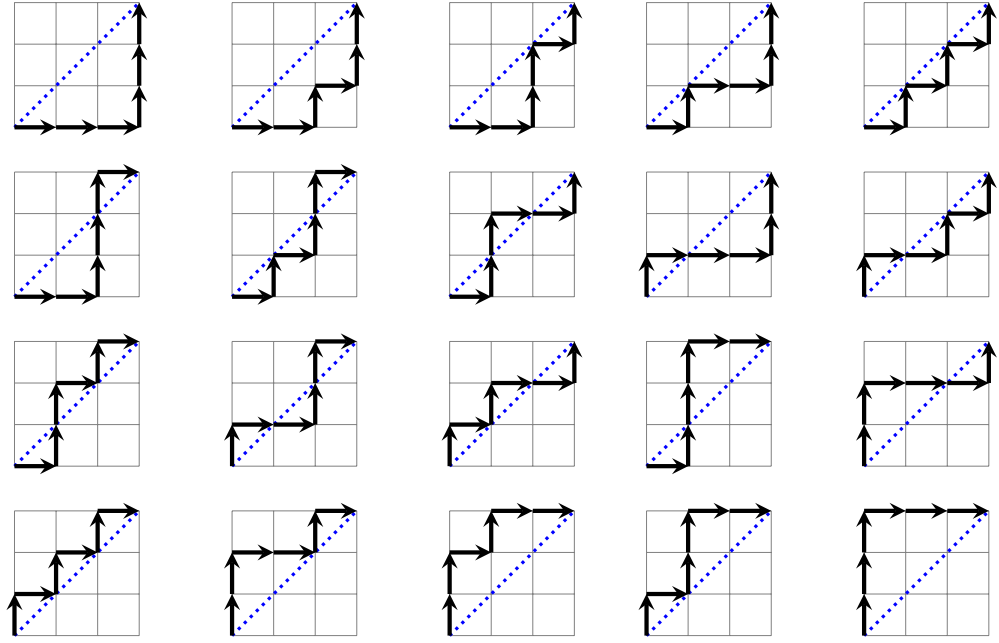

        \centering
        \begin{subfigure}{0.18\textwidth}
        \centering

        \end{subfigure}

        \caption{Set of partitions when $n=3$ with Chen's proof.}
        \label{chenpartitions}
    \end{figure}

\begin{figure}
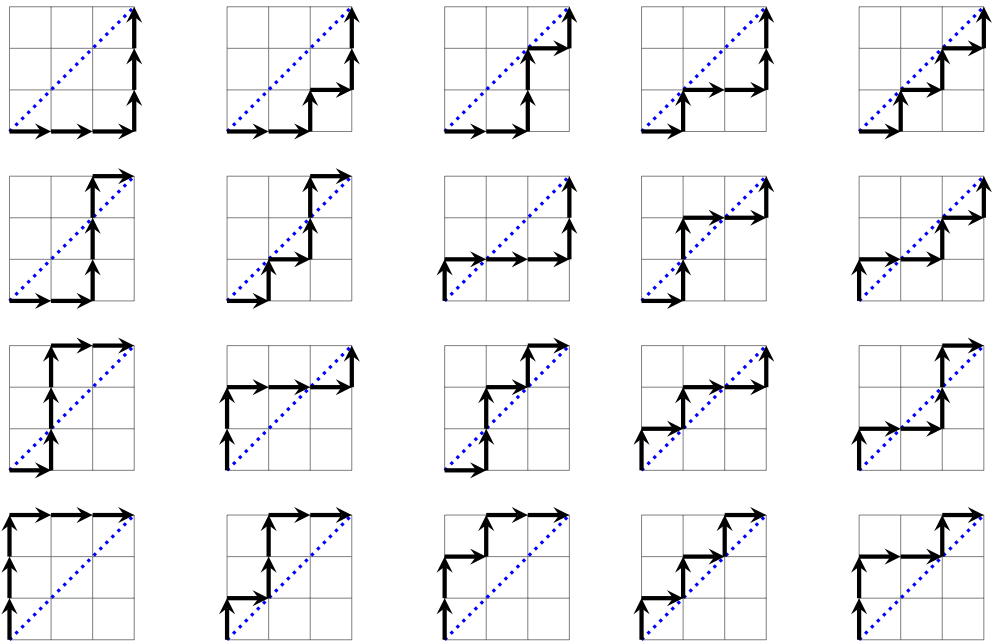

        \centering
        \begin{subfigure}{0.18\textwidth}
        \centering
        %
        \end{subfigure}

        \caption{Set of partitions when $n=3$ with the simple mirror variation of Chen's proof.}
        \label{chensimplemirror}
    \end{figure}

\begin{theorem}[Complex Mirror Variations of Chen's Proof]\label{chenmirrorcomplex}
    There is a bijection between the sets $S_k$ and $S_{k+1}$.
\end{theorem}

\begin{proof}
    In this case, we depict three new permutations of the paths, related to mirroring individual subpaths instead of the entire path. Firstly, Chen's bijection decomposes a lattice path with $k$ flaws into subpaths $BuAdC$, and reassembles them into a lattice path $AdBuC$, which he proves has $k+1$ flaws. Then, he shows this process is reversible.

    We propose that the following paths also have $k+1$ flaws: $\overline{A}dBuC,A\overline{dBu}C$, and $AdBu\overline{C}$. The crux of Chen's proof is as follows: by his construction, $A$ has zero flaws, $B$ has $k_1$ flaws, $dBu$ has $k_1+1$ flaws, and $C$ has the remaining flaws $k-k_1$ flaws, so the resulting concatenation has $k+1$ flaws. If we apply the mirror to either $A, dBu,$ or $C$ after applying the original Chen bijection, we want to get a path which does have $k+1$ flaws.
    
    By \ref{mubijection}, $\overline{A}$ will have the same number of flaws as $A$, and similarly for $\overline{dBu}$ and $\overline{C}$. By Chen's construction, each one of these subpaths touches the $x$-axis at both ends, so each of $\overline{A}dBuC,A\overline{dBu}C$, and $AdBu\overline{C}$ would have $k+1$ flaws. By \ref{uniqueness}, it would be possible to recover the respective portions of the path as Chen had in his proof, and then mirror whichever component of the path had been mirrored. So all three paths also have $k+1$ flaws, and since the mirror is reversible and Chen's reversible general construction, we can recover the original unique $k$ flaws.
\end{proof}

Since Chen's proof deals with subpaths, for small values of $n$, it can be hard to see that the added step of the mirror induces a noticeable change of partitions. For this case, we choose a larger $n$-value, or $n=7$ corresponding to a lattice path with length 14, and we give one partition induced by the multijection. In all five cases, we start with the same Dyck path, but show that the partitions begin to diverge after the mirror variation is applied over and over. We have the original bijections in Figure \ref{chenusualpartitions}, and then the complex mirror variation in Figures \ref{chencomplexmirror1}, \ref{chencomplexmirror2}, and \ref{chencomplexmirror3}.

\begin{figure}
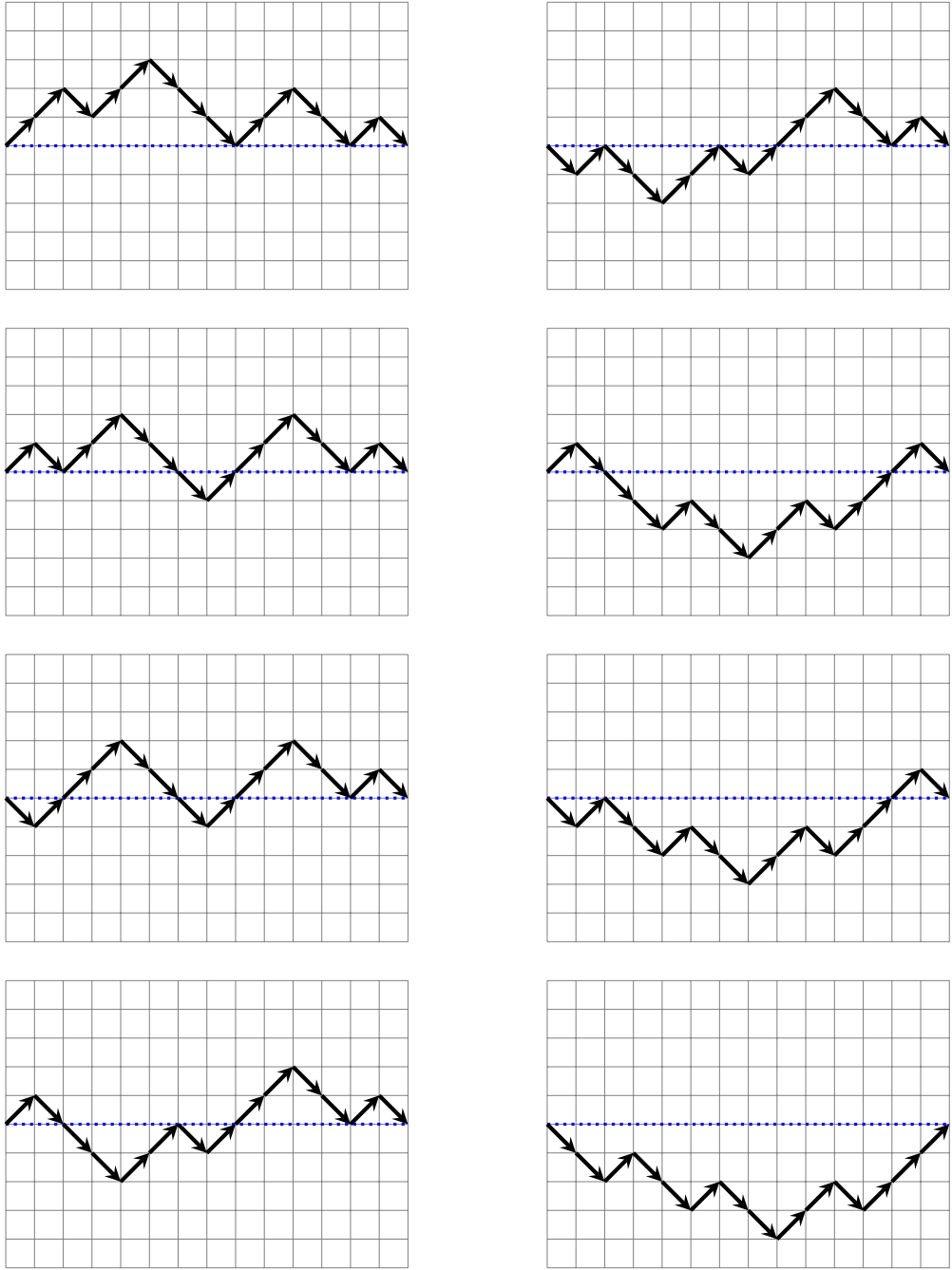

        \centering
        \begin{subfigure}{0.45\textwidth}
        \centering

        \end{subfigure}

        \caption{When the bijection maps $BuAdC\mapsto AdBuC$.}
        \label{chenusualpartitions}
    \end{figure}

\begin{figure}
        \centering
        \begin{subfigure}{0.45\textwidth}
        \centering
        %
        \end{subfigure}

        \caption{When the bijection is given as $BuAdC\mapsto \overline{A}dBuC$.}
        \label{chencomplexmirror1}
    \end{figure}

    \begin{figure}
        \centering
        \begin{subfigure}{0.45\textwidth}
        \centering
        %
        \end{subfigure}

        \caption{When the bijection is given as $BuAdC\mapsto A\overline{dBu}C$.}
        \label{chencomplexmirror2}
    \end{figure}

    \begin{figure}
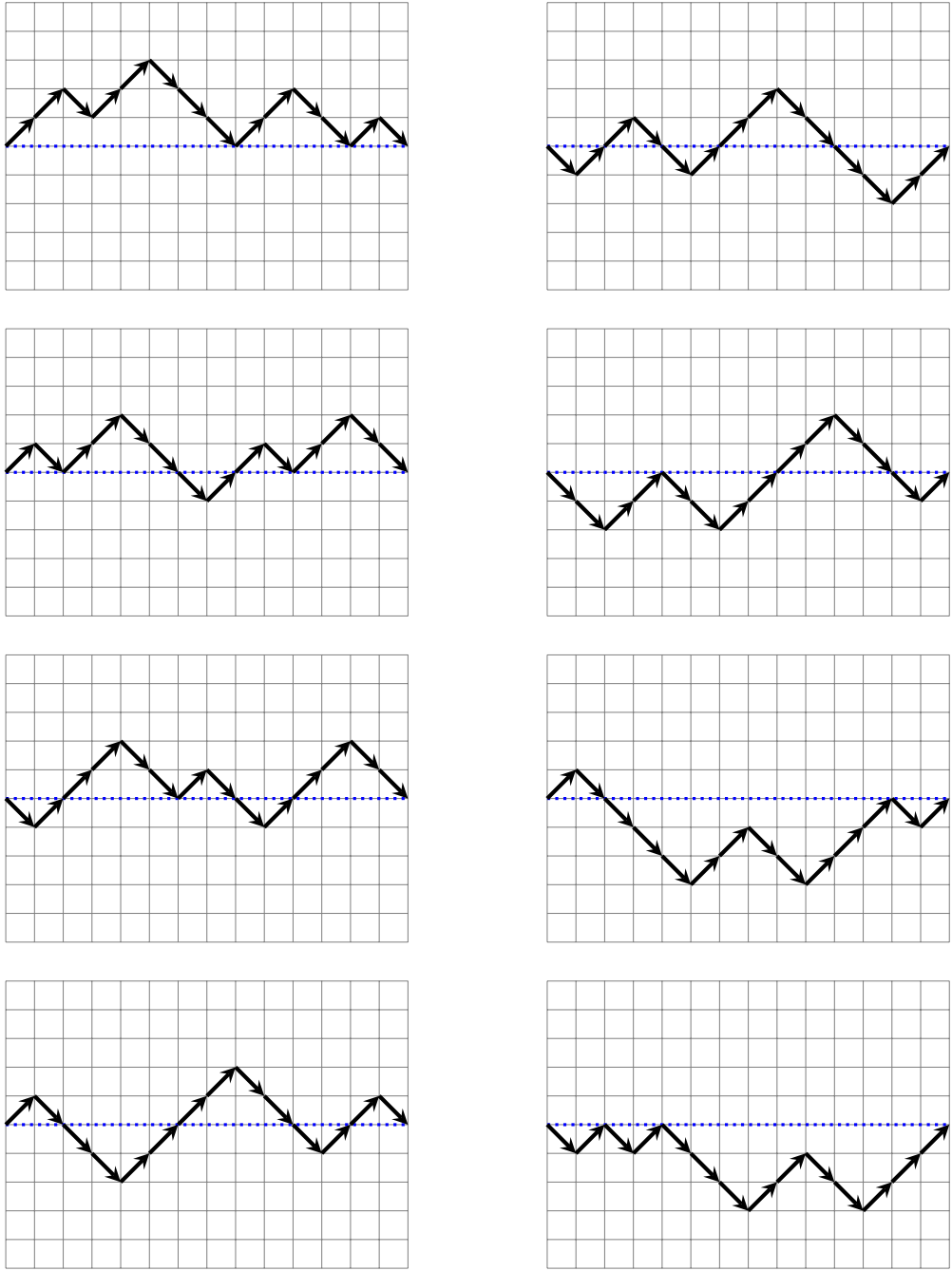

        \centering
        \begin{subfigure}{0.45\textwidth}
        \centering
        %
        \end{subfigure}

        \caption{Finally, when the bijection takes $BuAdC\mapsto AdBu\overline{C}$.}
        \label{chencomplexmirror3}
    \end{figure}

\section{Acknowledgements}

I would like to thank Garth Isaak for first introducing me to this topic in the summer of 2025 as part of a research project. Additionally, the precise definitions of a multijection, as well as conditions for their equivalence in \ref{multijection1} and \ref{equivalentmultijections} come directly from him.

Additionally, Figure \ref{chenpartitions} was crafted from the partitions directly listed in Chen's paper\cite{chen2008chung}. We give credit to him for writing out the partitions there, and the reproduction in this paper is simply to align with the formatting we have taken here.

\nocite{*} 
\bibliographystyle{abbrv}
\bibliography{references}

@misc{isaakmultijections,
    author    = "Garth Isaak and Larry Langley",
    title     = "\href{https://bpb-us-w2.wpmucdn.com/wordpress.lehigh.edu/dist/b/4296/files/2026/05/CatalanMultijectionsNew.pdf}{Catalan Multijections}",
    journal = "Unpublished",
    url       = {https://bpb-us-w2.wpmucdn.com/wordpress.lehigh.edu/dist/b/4296/files/2026/05/CatalanMultijectionsNew.pdf}
}

@article{callan1995pair,
  title={Pair them up! A visual approach to the Chung-Feller theorem},
  author={Callan, David},
  journal={The College Mathematics Journal},
  volume={26},
  number={3},
  pages={196--198},
  year={1995},
  publisher={Taylor \& Francis}
}

@article{chungfeller,
author = {Kai Lai Chung  and W. Feller },
title = {On Fluctuations in Coin-Tossing*},
journal = {Proceedings of the National Academy of Sciences},
volume = {35},
number = {10},
pages = {605-608},
year = {1949},
doi = {10.1073/pnas.35.10.605},
URL = {https://www.pnas.org/doi/abs/10.1073/pnas.35.10.605},
eprint = {https://www.pnas.org/doi/pdf/10.1073/pnas.35.10.605}}

@book{stanley2015catalan,
  title={Catalan numbers},
  author={Stanley, Richard P},
  year={2015},
  publisher={Cambridge University Press}
}

@article{rubenstein1994catalan,
  title={Catalan numbers revisited},
  author={Rubenstein, Daniel},
  journal={Journal of Combinatorial Theory, Series A},
  volume={68},
  number={2},
  pages={486--490},
  year={1994},
  publisher={Elsevier}
}

@article{woan2001uniform,
  title={Uniform partitions of lattice paths and Chung-Feller generalizations},
  author={Woan, Wen-jin},
  journal={The American Mathematical Monthly},
  volume={108},
  number={6},
  pages={556--559},
  year={2001},
  publisher={Taylor \& Francis}
}

@article{chen2008chung,
  title={The Chung-Feller theorem revisited.},
  author={Chen, Young-Ming},
  journal={Discrete Mathematics},
  volume={308},
  number={7},
  pages={1328--1329},
  year={2008},
  publisher={Elsevier Science}
}

@misc{poster,
  author       = {Dombrowski, Muhammad Adam and Le, Ivory},
  title        = {Catalan Multijections: which partitions are the same?},
  year         = {2025},
  publisher    = {Brown University Digital Repository},
  howpublished = {\url{https://repository.library.brown.edu/studio/item/bdr:zjjjsfb8/}},
  note         = {Accessed: May 25, 2026} 
}

\end{document}